\documentclass[11pt,reqno]{amsart}

\usepackage[utf8]{inputenc}
\usepackage[T1]{fontenc}
\usepackage[english]{babel}

\usepackage[margin=1in]{geometry}
\usepackage{xcolor}
\usepackage{amsmath}
\usepackage{amssymb}
\usepackage{amsthm}
\usepackage{mathrsfs}
\usepackage{tikz-cd}
\usepackage{enumitem}
\usepackage{caption}
\usepackage{graphicx}
\usepackage{hyperref}

\numberwithin{equation}{section}
\setlist[enumerate]{label=(\alph*)}
\setlist[itemize]{leftmargin=20pt}

\theoremstyle{plain}
\newtheorem{theorem}{Theorem}[section]
\newtheorem{proposition}[theorem]{Proposition}
\newtheorem{conjecture}[theorem]{Conjecture}
\newtheorem{lemma}[theorem]{Lemma}
\newtheorem{corollary}[theorem]{Corollary}
\theoremstyle{definition}
\newtheorem{definition}[theorem]{Definition}
\newtheorem{construction}[theorem]{Construction}

\newtheorem*{convention}{Convention}
\newtheorem{remark}[theorem]{Remark}
\newtheorem{example}[theorem]{Example}
\newtheorem{question}[theorem]{Question}

\renewcommand{\epsilon}{\varepsilon}

\title{Uniqueness of $p$-local truncated Brown-Peterson spectra}
\author{David Jongwon Lee}
\address{Department of Mathematics, MIT, Cambridge, MA, USA}
\email{jongwonl@mit.edu}

\newcommand{\Map}{\mathrm{Map}}

\newcommand{\Ext}{\operatorname{Ext}}

\newcommand{\cof}{\operatorname{cof}}
\newcommand{\Hom}{\operatorname{Hom}}

\newcommand{\Seq}{\operatorname{Seq}}
\newcommand{\tmf}{\mathrm{tmf}}
\newcommand{\ko}{\mathrm{ko}}
\newcommand{\ku}{\mathrm{ku}}
\newcommand{\BP}{\mathrm{BP}}
\newcommand{\gl}{\mathrm{gl}}
\newcommand{\textbpn}{\texorpdfstring{$\BP\langle n\rangle$}{BP<n>}}
\newcommand{\textp}{\texorpdfstring{$p$}{p}}
\newcommand{\Maj}{\operatorname{Maj}}
\newcommand{\ind}{\operatorname{ind}}
\newcommand{\BSO}{\mathrm{BSO}}
\newcommand{\BSU}{\mathrm{BSU}}

\begin{document}
    \begin{abstract}
        When $p$ is an odd prime, we prove that the $\mathbb F_p$-cohomology of $\mathrm{BP}\langle n\rangle$ as a module over the Steenrod algebra determines the $p$-local spectrum $\mathrm{BP}\langle n\rangle$. In particular, we prove that the $p$-local spectrum $\mathrm{BP}\langle n\rangle$ only depends on its $p$-completion $\mathrm{BP}\langle n\rangle_p^\wedge$. As a corollary, this proves that the $p$-local homotopy type of $\mathrm{BP}\langle n\rangle$ does not depend on the ideal by which we take the quotient of $\mathrm{BP}$. In the course of the argument, we show that there is a vanishing line for odd degree classes in the Adams spectral sequence for endomorphisms of $\mathrm{BP}\langle n\rangle$. We also prove that there are enough endomorphisms of $\mathrm{BP}\langle n\rangle$ in a suitable sense. When $p=2$, we obtain the results for $n\leq 3$.
    \end{abstract}
        \maketitle
    \tableofcontents
    \section{Introduction}
    \subsection{Main results}
    The Brown-Peterson spectrum $\BP$, defined for each prime $p$, is a central object in $p$-local chromatic homotopy theory. It is a structured ring spectrum (\cite{CM}, \cite{BM}), and its homotopy groups are
    \[
        \pi_\ast\BP = \mathbb Z_{(p)}[v_1,v_2,\dots]
    \]
    with $|v_i|=2p^i-2$. As a graded ring, $\pi_\ast\BP$ is the universal ring with a $p$-typical formal group law. Roughly speaking, chromatic homotopy theory is the study of the $p$-local stable homotopy category via the stratification of the moduli of formal groups by heights, where the closed stratum of height $> n$ is defined by the equations $p=v_1=\cdots=v_{n}=0$.

    Given a nonnegative integer $n$, the $\BP$-module
    \[
        \BP\langle n\rangle = \BP/(v_{n+1},v_{n+2},\dots)
    \]
    is one of the most basic objects in the theory. In this paper, we show that the $p$-local spectrum $\BP\langle n\rangle$ is uniquely characterized by its $\mathbb F_p$-cohomology. Since the $\mathbb F_p$-cohomology of a spectrum only depends on its $p$-completion, the claim is naturally a combination of the following two theorems. Let us say that a $p$-local spectrum is \emph{of finite type} if it is bounded below and each homotopy group is a finitely generated $\mathbb Z_{(p)}$-module.
    \begin{theorem}\label{thm:intro1}
        Let $p$ be an odd prime and let $X$ be a $p$-local spectrum of finite type. Then, any isomorphism
        \[
            H^\ast(X;\mathbb F_p)\simeq H^\ast(\BP\langle n\rangle;\mathbb F_p)
        \]
        as modules over the Steenrod algebra can be lifted to an equivalence of spectra $X_p^\wedge\simeq \BP\langle n\rangle_p^\wedge$. The same result holds when $p=2$ if $n\leq3$.
    \end{theorem}
    
    \begin{theorem}\label{thm:introplocal}
        Let $p$ be an odd prime and let $X$ be a $p$-local spectrum of finite type. Then, given any equivalence $f:X_p^\wedge\simeq \BP\langle n\rangle_p^\wedge$, there is an equivalence $X\simeq \BP\langle n\rangle$ that lifts $f/p:X/p\simeq \BP\langle n\rangle/p$. When $p=2$, the theorem holds if $n\leq 3$.
    \end{theorem}

    The following key result about the $E_2$-page of the Adams spectral sequence that computes maps from $\BP\langle n\rangle_p^\wedge$ to $\BP\langle m\rangle_p^\wedge$ immediately implies Theorem \ref{thm:intro1} by considering the case $m=n$. It will also be crucially used in the proof of Theorem \ref{thm:introplocal}.
    \begin{theorem}\label{thm:intro6}
        Let $p$ be an odd prime, $n\geq m$ nonnegative integers, and $\mathcal A$ the mod $p$ Steenrod algebra. Then, there exists a positive real number $\epsilon>0$, depending only on $p$, $m$, and $n$, such that
        \[
            \Ext_{\mathcal A}^{s,t}(H^\ast(\BP\langle m\rangle;\mathbb F_p),H^\ast(\BP\langle n\rangle;\mathbb F_p))
        \]
        is zero if $t-s$ is odd and $\epsilon s + (t-s) > 1-2p^{n+1}$. Informally, there is a line of negative slope in the Adams chart that meets the $x$-axis at $(1-2p^{n+1},0)$ so that there are no odd degree classes above that line.  When $p=2$, the theorem holds if either $n=m\leq 2$ or $n>m$.
    \end{theorem}

    \begin{remark}
        Let us briefly discuss the history of this problem. The case $n=1$ of Theorems \ref{thm:intro1} and \ref{thm:introplocal} is a classic result of Adams and Priddy \cite{AP}. Their argument does not clearly separate Theorem \ref{thm:introplocal} from the combined theorem, and it required a careful analysis of now-standard convergence issues in Adams spectral sequences. One goal of this paper is to clarify their argument while generalizing it to arbitrary $n$. A special feature of the case $n=1$ is the existence of geometrically defined Adams operations, for which we must find a suitable substitute in our generalization.

        More recently, in their work \cite{AL}, Angeltveit and Lind showed that Theorem \ref{thm:intro6} holds for $n=m$, which immediately implies Theorem \ref{thm:intro1}. Furthermore, in \cite[Remark 4.1]{AL}, they observed that in order to prove Theorem \ref{thm:introplocal}, one would have to produce many automorphisms of $\BP\langle n\rangle$ as we do in Section \ref{sec:uniqueplocal}.

        Unfortunately, there appears to be an error in a part of their argument where they obtain upper bounds on the Adams filtration of some Ext classes. See Remark \ref{rmk:AL}. In private communication, Angeltveit independently suggested a fix that is similar to the argument given in this paper, so Theorem \ref{thm:intro6} and thus Theorem \ref{thm:intro1} should be considered due to Angeltveit and Lind. The new argument is not strong enough for the prime $2$, so even Theorem \ref{thm:intro1} is still an open question for $n\geq 4$.

        Theorem \ref{thm:introplocal} is entirely new and surprisingly subtle despite its innocuous look. For example, it is false for $\BP\langle n\rangle\oplus\BP\langle m\rangle$ if $m\neq n$. See Example \ref{ex:Zplusell} for a discussion on $\BP\langle 0\rangle\oplus\BP\langle1\rangle$.
    \end{remark}
    
    \begin{conjecture}\label{conj:intro}
        Theorems \ref{thm:intro1}, \ref{thm:introplocal}, and \ref{thm:intro6} are true at $p=2$ for all $n$ and $m$.
    \end{conjecture}
    
    As stated in the conjecture above, the author expects every theorem in this paper to hold for all $n$ and $m$ when $p=2$, but the methods of this paper are not sufficient for the remaining cases.

    The following is a statement about maps between $p$-local $\BP\langle n\rangle$ and $\BP\langle m\rangle$ that may be of interest. It is proved in Section \ref{sec:uniqueplocal}. See also Lemma \ref{lem:plocallift} for the description of the image of the following injection.
    \begin{proposition}\label{thm:introinj}
        Let $p$ be an odd prime and let $n\geq m$ be nonnegative integers. Then, for any integer $k$, the natural map
        \[
            [\BP\langle n\rangle,\Sigma^k\BP\langle m\rangle]\to [\BP\langle n\rangle_p^\wedge,\Sigma^k\BP\langle m\rangle_p^\wedge]
        \]
        is injective. In particular, combined with Theorem \ref{thm:intro6}, we have
        \[
            [\BP\langle n\rangle,\Sigma^k\BP\langle m\rangle]=[\BP\langle n\rangle,\Sigma^k\BP\langle m\rangle_p^\wedge]=0
        \]
        for $k$ odd and strictly less than $2p^{n+1}-1$. The same is true when $p=2$ if either $n>m$ or $n=m\leq 2$.
    \end{proposition}

    Finally, we note that in Section \ref{sec:uniqueplocal} we will actually prove the following stronger version of the main theorems.

    \begin{theorem}\label{thm:intro1s}
        Let $p$ be an odd prime, $X$ a $p$-local spectrum of finite type, and $Y$ a direct sum of even suspensions of $\BP\langle n\rangle$
        \[
            Y = \Sigma^{2i_1}\BP\langle n\rangle\oplus\cdots\oplus\Sigma^{2i_k}\BP\langle n\rangle
        \]
        such that $0\leq i_1\leq\cdots\leq i_k < p^{n+1}-1$. Then, any isomorphism
        \[
            H^\ast(X;\mathbb F_p)\simeq H^\ast(Y;\mathbb F_p)
        \]
        as modules over the Steenrod algebra can be lifted to an equivalence of spectra $X\simeq Y$. When $p=2$, the same result holds if $n\leq 2$.
    \end{theorem}

    \subsection{Applications}
    Let us describe three examples where Theorem \ref{thm:intro1s} can be applied. In each of the applications, the $\mathbb F_p$-cohomology of a spectrum is computed by other means, and we use the theorem to determine its homotopy type.

    First, the following corollary is immediate.
    \begin{corollary}
        Let $p$ be an odd prime. Then, the homotopy type of $\BP\langle n\rangle$ does not depend on the choice of generators $v_{n+1},v_{n+2},\dots\in\pi_\ast\BP$. When $p=2$, the same holds if $n\leq3$ and conjecturally for all $n$ depending on Conjecture \ref{conj:intro}.
    \end{corollary}
    Recall that there is no preferred choice of the generators $v_1,v_2,\dots$ of $\pi_\ast\BP$, while the $\BP$-module structure of $\BP\langle n\rangle$ clearly depends on the choice of generators, or more precisely on the ideal $(v_{n+1},v_{n+2},\dots)$. However, the corollary follows from the fact that the $\mathbb F_p$-cohomology of $\BP\langle n\rangle$ as a module over the Steenrod algebra does not depend on the choice of generators. See Section \ref{sec:HBPn} for a detailed discussion of its cohomology.
    

    The second application, which appears in the original work \cite{AP} of Adams and Priddy, is the identification of the spectrum $\gl_1(\ku)$. In \cite[Theorem 1.2]{AP}, they show that if $X$ is a $p$-local spectrum such that $\Omega^\infty X = BSU_{(p)}$, then $H^\ast(X;\mathbb F_p)=H^\ast(\Sigma^4\ku;\mathbb F_p)$ as modules over the Steenrod algebra. Then, the equivalences
    \[
        \Omega^{\infty}\tau_{\geq 3} \gl_1(\ku)_{(p)} = \Omega^{\infty}\tau_{\geq 3} \ku_{(p)} =  \BSU_{(p)}
    \]
    imply that $H^\ast(\tau_{\geq_3}\gl_1(\ku);\mathbb F_2)= H^\ast(\Sigma^4\ku;\mathbb F_2)$. Therefore, by Theorem \ref{thm:intro1s}, we have an equivalence
    \[
        \tau_{\geq3}\gl_1(\ku)_{(p)} = \Sigma^4\ku_{(p)},
    \]
    and the identification of $\gl_1(\ku)$ reduces to the attaching maps of two more homotopy groups. New applications of results of this paper to the $\gl_1$-spectra of $\mathbb E_\infty$-ring spectra of chromatic height $2$ will appear in a forthcoming paper by Jeremy Hahn and Andrew Senger.

    The last application we discuss is the splitting of $\tmf$, the spectrum of topological modular forms, at primes $p\geq 5$. See, for example, \cite{tmfsurvey} for a survey of $\tmf$. The following theorem seems to be a folklore and well-known to experts, but the proof does not exist in the literature. Here, we record it as one of the applications of Theorems \ref{thm:intro6} and \ref{thm:intro1s}. The computations  of the cohomology groups of $\tmf$ appear in \cite[Theorem 21.5]{Rezktmf}. See Corollary \ref{cor:tmfsplit} for details.
    \begin{theorem}\label{thm:introtmf}
        The spectrum of topological modular forms $\tmf$ splits into a direct sum of even suspensions of $\BP\langle 2\rangle$ after localization at a prime $p\geq5$.
    \end{theorem}

    \subsection{Proof outline}
    The proof of Theorem \ref{thm:intro1} is a standard Adams spectral sequence argument. More precisely, there is a strongly convergent spectral sequence (Remark \ref{rmk:conv})
    \[
        E_2^{s,t}=\Ext_{\mathcal A}^{s,t}(H^\ast(\BP\langle n\rangle;\mathbb F_p),H^\ast(X;\mathbb F_p))\Rightarrow[\Sigma^{t-s}X_p^\wedge,\BP\langle n\rangle_p^\wedge]
    \]
    where $\mathcal A$ is the mod $p$ Steenrod algebra. If the group $E_2^{s,t}$ vanishes when $t-s=-1$, then the isomorphism $H^\ast(\BP\langle n\rangle;\mathbb F_p)\simeq H^\ast(X;\mathbb F_p)$ is not the source of any differential in the spectral sequence, and it can be lifted to an equivalence $X_p^\wedge\simeq \BP\langle n\rangle_p^\wedge$. In particular, it follows from the stronger Theorem \ref{thm:intro6}, which we prove in Section \ref{sec:uniquepadic}. We note that Theorem \ref{thm:intro6} was proved for $n=m=1$ in \cite[Proposition 4.1]{AP} by completely computing the Ext groups.

    For the proof of Theorem \ref{thm:introplocal}, we prove the following result about obtaining $p$-local homotopy types from $p$-complete ones. This argument is implicitly used in \cite[Section 5]{AP}.
    \begin{theorem}\label{thm:intro3}
        Let $X, Y$ be $p$-local spectra of finite type. Suppose also that the natural map
        \[
            [Y,Y]\otimes_{\mathbb Z}\mathbb Q\to \prod_{i\leq k}\Hom_{\mathbb Q}(\pi_i(Y)\otimes_{\mathbb Z}\mathbb Q,\pi_i(Y)\otimes_{\mathbb Z}\mathbb Q)
        \]
        is surjective for all $k$. Then, for any equivalence $f:X_p^\wedge\simeq Y_p^\wedge$, there exists an equivalence $X\simeq Y$ lifting $f/p:X/p\simeq Y/p$.
    \end{theorem}

    It is easy to see that $Y=\BP\langle 1\rangle$ satisfies the assumption of the theorem using Adams operations (Example \ref{ex:BP1AP}). We shall prove that it is true for $n>1$ in general when $p$ is odd. Informally, it states that there are many endomorphisms of $\BP\langle n\rangle$ that are nontrivial on homotopy groups, and we believe this paper is the first to produce such results.
    \begin{theorem}\label{thm:intro4}
        The assumption on $Y$ of theorem \ref{thm:intro3} is satisfied by $Y=\BP\langle n\rangle$ when $p$ is odd. When  $p=2$, the same holds for $n\leq 3$ and conjecturally for all $n$ depending on the validity of Theorem \ref{thm:intro6}.
    \end{theorem}

    More precisely, we construct maps from $\BP\langle n\rangle$ to $\BP\langle m\rangle$ for $n\geq m$ as in the following theorem, which may be of independent interest.
    \begin{theorem}\label{thm:intro5}
        Suppose that $(p,m,n)$ satsifies Theorem \ref{thm:intro6}. Let $I=(i_1,\dots,i_n)$ be a sequence of nonnegative integers of length $n$, and let
        \[
            D := \deg(v_1^{i_1}\cdots v_n^{i_n}).
        \]
        Then, for any sufficiently large integer $N$, there is a map\[
        \phi^I:\BP\langle n\rangle\to\Sigma^{D}\BP\langle m\rangle\]
        that sends $v_1^{i_1}\cdots v_n^{i_n}$ to $p^N$ and other monomials of degree $D$ to $0$.
    \end{theorem}

    \begin{remark}
        We note again that it is not always true that one can recover the $p$-local homotopy type from its $p$-completion. In Example \ref{ex:Zplusell}, we shall see that there exists a $p$-local spectrum of finite type which is not equivalent to $\BP\langle0\rangle\oplus\BP\langle 1\rangle$ but becomes equivalent after $p$-completion.
    \end{remark}

    \subsection{Further questions}
    The $E_2$-page of the Adams spectral sequence for maps from $\BP\langle n\rangle$ to $\BP\langle m\rangle$ can be computed explicitly when $m=1$ using \cite[Table 3.9]{AP}. 
    \begin{question}
        What is the optimal (i.e. largest) value of $\epsilon$ in Theorem \ref{thm:intro6}? Explicit computation tells us that when $n=m=1$, the optimal value of $\epsilon$ is $2(p-1)^2$.
    \end{question}

    \begin{question}
        For $n\geq m$, does the Adams spectral sequence
        \[
            E_2 = \Ext_{\mathcal A}^{s,t}(H^\ast(\BP\langle m\rangle;\mathbb F_p),H^\ast(\BP\langle n\rangle))\Rightarrow [\Sigma^{t-s}\BP\langle n\rangle_p^\wedge,\BP\langle m\rangle_p^\wedge]
        \]
        degenerate at the $E_2$-page?
    \end{question}
    
    Adams and Priddy in \cite[Theorem 1.1]{AP} also prove the analogues of Theorems \ref{thm:intro1} and \ref{thm:introplocal} for $\ko$, the connective real $K$-theory spectrum. More precisely, they prove that if $X$ is a $2$-local spectrum of finite type such that $H^\ast(X;\mathbb F_2)\simeq H^\ast(\ko;\mathbb F_2)$ as modules over the Steenrod algebra, then $X$ is equivalent to $\ko_{(2)}$. This result is used in their paper to prove that $\BSO_{(2)}$ has a unique $\mathbb E_\infty$-space structure, which is applied to the identification of $\gl_1(\ko)$.

    Similarly, as previously asked by Devalapurkar in \cite[Conjecture F]{Sanath}, it would be interesting to know if $\tmf$ or other related spectra satisfies similar properties.
    \begin{question}
        Consider the Adams spectral sequence
        \[
            E_2^{s,t}=\Ext_{\mathcal A}^{s,t}(H^\ast(\tmf;\mathbb F_p),H^\ast(\tmf;\mathbb F_p)) \Rightarrow [\Sigma^{t-s}\tmf_p^\wedge,\tmf_p^\wedge].
        \]
        Then, for each integer $s\geq 2$, is there a page $r$ such that $E_r^{s,s-1}=0$?
    \end{question}
    For $p>3$, we will give a positive answer to this question since $\tmf_{(p)}$ is a sum of shifts of $\BP\langle 2\rangle$  by Theorem \ref{thm:introtmf}. To recover the $p$-local homotopy type from the $p$-complete homotopy type of $\tmf$ as in Theorem \ref{thm:introplocal}, we will prove the following theorem in Section \ref{sec:AParg}.
    \begin{theorem}
        For any prime $p$, the assumption on $Y$ of Theorem \ref{thm:intro3} is satisfied for $Y=\tmf_{(p)}$.
    \end{theorem}

    Recall that there is a genuine $C_2$-spectrum $\BP\langle n\rangle_{\mathbb R}$ induced by the complex conjugation action on the complex cobordism spectrum \cite{HK}. It would be interesting to know if similar statements can be made about $\BP\langle n\rangle_{\mathbb R}$.
    \begin{question} 
        Can we state and prove similar theorems for $\BP\langle n\rangle_{\mathbb R}$?
    \end{question}

    Combining Theorems \ref{thm:intro3}, \ref{thm:BPnAP} and Proposition \ref{prop:tensorfin}, we shall see in the paper that $\BP\langle n\rangle\otimes A$ for any finite spectrum $A$ has a unique $p$-local model, i.e. satsifies the analogue of  Theorem \ref{thm:introplocal}. It is also easy to see that an analogue of Theorem \ref{thm:introplocal} is satisfied by finite $p$-local spectra or bounded spectra. It would be interesting to know if Theorem \ref{thm:introplocal} is true for a much braoder class of spectra, given the interesting counterexample in Example \ref{ex:Zplusell}.
    \begin{question}
         Can we find a large class of spectra that satisfies the analogue of Theorem \ref{thm:introplocal}?
    \end{question}
    
    \subsection{Acknowledgements}
    I would like to thank Mark Behrens, Jeremy Hahn, and Ishan Levy helpful conversations related to this work. I would especially like to thank Vigleik Angeltveit for an enlightening discussion on his original work with John Lind.
    \subsection{Notations and conventions}
    \begin{itemize}
        \item $p$ is a prime number fixed throughout the paper and $\mathcal A$ denotes the mod $p$ Steenrod algebra.
        \item A $p$-local spectrum $X$ is said to be of finite type if it is bounded below and every homotopy group $\pi_i X$ is finitely generated over $\mathbb Z_{(p)}$.
        \item Given two spectra $X$ and $Y$, we write $\Map(X,Y)$ for the mapping space, $F(X,Y)$ for the function spectrum from $X$ to $Y$, and $[X,Y]$ for $\pi_0F(X,Y)$.
        \item The notation $\Lambda(x_0,\dots,x_n)$ denotes an exterior algebra over $\mathbb F_p$ generated by $x_0,\dots,x_n$.
        \item All sequences of integers in the paper, whether of finite length or infinite length, should be considered as elements of $\mathbb Z^{\oplus\mathbb N}$. In particular, all sequences are assumed to have finitely many nonzero entries, and additions and subtractions should be done with the group structure of $\mathbb Z^{\oplus\mathbb N}$.
        
        We write
        \begin{align*}
            \mathbf f_0=\mathbf e_1 &= (1,0,0,\dots)\\
            \mathbf f_1=\mathbf e_2 &= (0,1,0,\dots)\\
            &\cdots
        \end{align*}
        for basic integer vectors. We shall use $\mathbf e_i$'s for $1$-base indexed sequences and $\mathbf f_i$'s for $0$-base indexed sequences. Unfortunately, both situations occur frequently.
    \end{itemize}
    \section{Cohomology of \textbpn}\label{sec:HBPn}
    \subsection{Steenrod algebra}
    In this section, we review the structure of the Steenrod algebra in Milnor's work \cite{Milnor}.
    
    The dual Steenrod algebra $\pi_\ast(\mathbb F_p\otimes\mathbb F_p)$ is a graded commutative $\mathbb F_p$-Hopf algebra which is isomorphic as an algebra to
    \[
        \mathbb F_p[\overline{\xi_1},\overline{\xi_2},\dots]\otimes_{\mathbb F_p}\Lambda(\overline{\tau_0},\overline{\tau_1},\dots)
    \]
    with degrees $|\overline{\xi_i}|=2p^i-2$ and $|\overline{\tau_i}|=2p^i-1$. The generators $\overline{\xi_1},\overline{\xi_2},\dots,\overline{\tau_0},\overline{\tau_1}$ are the conjugate Milnor generators. When $p=2$, we have $\overline{\tau_i}^2 = \overline{\xi_{i+1}}$ instead of $\overline{\tau_i}^2=0$, but this will not affect any of our arguments.

    The Steenrod algebra $\mathcal A=\pi_\ast F(\mathbb F_p,\mathbb F_p)$ is the dual Hopf algebra of the dual Steenrod algebra. For a sequence of nonnegative integers $I=(i_1,i_2,\dots)$ and a sequence of $0$'s and $1$'s $E=(\epsilon_0,\epsilon_1,\dots)$, let us write
    \[
        \{P^IQ^E\}_{I,E}
    \]
    for the dual basis to the monomial basis
    \[
        \{\overline{\xi_1}^{i_1}\overline{\xi_2}^{i_2}\cdots\overline{\tau_0}^{\epsilon_0}\overline{\tau_1}^{\epsilon_1}\cdots\}_{I,E}.
    \]
    Also, let us write $Q_j$ for the dual of $\overline{\tau_j}$. We shall loosely refer to a basis element $P^IQ^E$ as a \emph{monomial}. By the following lemma, the notation $P^IQ^E$ is less ambiguous since it is actually the product of $P^I$ and $Q^E$.
  
    \begin{lemma}[\cite{Milnor}]\label{lem:steenrod}
        \begin{enumerate}[leftmargin=20pt]
            \item We have
            \[
                P^IQ^E = P^I \cdot Q_0^{\epsilon_0} \cdot Q_1^{\epsilon_1}\cdots
            \]
            where $\cdot$ denotes the multiplication of $\mathcal A$.
            \item The element $Q_k$ is primitive, and the subalgebra generated by $Q_0,Q_1,\dots$ is an exterior algebra $\Lambda(Q_0,Q_1,\dots)$.
            \item We have a formula
            \begin{align*}
                Q_k P^I &= P^IQ_k+\sum_{j=1}^\infty P^{I-p^k \mathbf e_j}Q_{k+j}\\
                &=P^IQ_k + P^{I-(p^k,0,\dots)}Q_{k+1} + P^{I-(0,p^k,0,\dots)}Q_{k+2}+\cdots
            \end{align*}
            for the left multiplication by $Q_k$'s. Here, if a sequence in a superscript of $P$ contains a negative number, then that whole term should read zero.
        \end{enumerate}
    \end{lemma}

    The monomial $P^{(i_1,i_2,\dots)}Q_0^{\epsilon_0}Q_1^{\epsilon_1}\cdots$ has degree
    \[
        -\sum_{j\geq 1} (2p^{j}-2) i_j - \sum_{j\geq 0} (2p^{j}-1)\epsilon_j
    \]
    in the homological convention. We recall another grading of the Steenrod algebra that will be essential in this paper.
    \begin{definition}\label{def:weight}
    The \emph{weight} of a monomial $P^{(i_1,i_2,\dots)}Q_0^{\epsilon_0}Q_1^{\epsilon_1}\cdots$ is defined to be
    \[
        \sum_{j\geq 1} 2p^{j}  i_j+ \sum_{j\geq 0} 2p^{j} \epsilon_j.
    \]
    \end{definition}
    
    \subsection{Cohomology of \textbpn}
    In this section, we review the cohomology of $\BP\langle n\rangle$ and study several gradings and filtrations on this cohomology group.
    \begin{definition}
        A \emph{form of $\BP\langle n\rangle$} is a $\BP$-module 
        \[
            \BP/(v_{n+1},v_{n+2},\dots)
        \]
        for some choice of generators $v_1,v_2,\dots\in\pi_\ast\BP$.
    \end{definition}
    
    Let $E(n)\subseteq\mathcal A$ be the subalgebra $\Lambda(Q_0,\dots,Q_n)$ of $\mathcal A$. The $\mathbb F_p$-cohomology of $\BP\langle n\rangle$ is
    \[
        H^\ast(\BP\langle n\rangle;\mathbb F_p) = \mathcal A//E(n) := \mathcal A\otimes_{E(n)}\mathbb F_p
    \]
    as a left $\mathcal A$-module regardless of the form \cite[Proposition 1.7]{WilsonBP}. By Lemma \ref{lem:steenrod}(a), it is generated, as an $\mathbb F_p$-vector space, by the elements $P^IQ^{(\epsilon_0,\epsilon_1,\dots)}$ such that $\epsilon_j=0$ for $j\leq n$. We shall mainly be interested in the left $E(n)$-module structure of $\mathcal A//E(n)$.

    \begin{lemma}\label{lem:formula}
        For $0\leq k\leq n$, the left action of $Q_k$ on $\mathcal A//E(n)$ is given by the formula
        \[
            Q_k \cdot P^I = P^{I-p^k\mathbf e_{n+1-k}} Q_{n+1} + P^{I-p^k\mathbf e_{n+2-k}} Q_{n+2}+\cdots
        \]
        extended right-$\Lambda(Q_{n+1},Q_{n+2},\dots)$-linearly.
        
        The left action of $Q_k$ for $0\leq k\leq n$ preserves the weight grading of $\mathcal A//E(n)$. Thus, $\mathcal A//E(n)$ splits into a direct sum of $E(n)$-modules according to the weights and each summand is a finite-dimensional $\mathbb F_p$-vector space.
    \end{lemma}
    \begin{proof}
        The formula follows from Lemma \ref{lem:steenrod}(c), and we can check that the right hand side has the same weight as $P^I$. Indeed, $P^{I-p^k\mathbf e_{n+\ell-k}}$ has weight $p^k(2p^{n+\ell-k})=2p^{n+\ell}$ less than that of $P^I$, which is compensated by $Q_{n+\ell}$. The last assertion is obvious.
    \end{proof}
    This action satisfies the following exactness property.
    \begin{proposition}[{\cite[Lemma 3.3]{AL}}]\label{prop:exact}
        Let $x\in\mathcal A//E(n)$ be an element of odd degree, and let $0\leq k\leq n$ be an integer. If $Q_k x = 0$, then there exists an element $y\in\mathcal A//E(n)$ such that $x=Q_ky$.
    \end{proposition}

    \begin{definition}\label{def:ellweight}
        Let $\ell\geq 1$ be a positive integer. The \emph{$\ell$-mixed weight} of a monomial $P^IQ^E\in\mathcal A$ with $I=(i_1,i_2,\dots)$ is defined to be
        \[
            \sum_{j<\ell} 2p^ji_j + 2p^\ell\sum_{j\geq \ell} i_j,
        \]
        which defines a grading on $\mathcal A$ and on $\mathcal A//E(n)$. It also defines an increasing filtration
        \[
            F_0^{n,\ell}\subseteq F_1^{n,\ell}\subseteq\cdots\subseteq\mathcal A//E(n)
        \]
        where $F_s^{n,\ell}$ is the sub-$\mathbb F_p$-vector space generated by monomials of $\ell$-mixed weight $\leq s$.
    \end{definition}

    \begin{lemma}\label{lem:ellweightineq}
        If $k\leq n-\ell$, then the $Q_k$-action on $\mathcal A//E(n)$ decreases the $\ell$-mixed weight by precisely $2p^{k+\ell}$. If $n-\ell<k\leq n$, then we have an inclusion
        \[
            Q_k(F_s^{n,\ell}) \subseteq F^{n,\ell}_{s-2p^{n+1}}
        \]
        for all $s$.
    \end{lemma}
    \begin{proof}
        In the formula of Lemma \ref{lem:formula}, the $t$-th term on the right hand side has $\ell$-mixed weight $2p^{k+s}$ less than that of $P^I$ where 
        \[
            s=\begin{cases}
            n+t-k&\text{if }n+t-k<\ell\\
            \ell&\text{if }n+t-k\geq \ell.
            \end{cases}
        \]
        by the definition of $\ell$-mixed weight.
        
        If $k\leq n-\ell$, then it is always the case that $n+t-k\geq\ell$, so the $\ell$-mixed weight decreases precisely by $2p^{k+\ell}$.
        
        Suppose $k>n-\ell$. Then, we have
        \[
            2p^{k+s} =\begin{cases}
            2p^{n+t}&\text{if }n+t-k<\ell\\
            2p^{k+\ell}&\text{if }n+t-k\geq \ell.
            \end{cases}
        \]
        Either case, we have $2p^{k+s}\geq2p^{n+1}$. This proves the Lemma.
    \end{proof}

    \begin{lemma}\label{lem:ellweightub}
        Suppose $x\in\mathcal A//E(n)$ is an element of odd degree and of homogeneous weight $w$. Then, $x\in F_{w-2p^{n+1}}^{n,\ell}$ for all $\ell$. In particular $w\geq 2p^{n+1}$.
    \end{lemma}
    \begin{proof}
        Let
        \[
        P^{(i_1,i_2,\dots)}Q_{n+1}^{\epsilon_{n+1}}Q_{n+2}^{\epsilon_{n+2}}\cdots
        \]
        be a monomial of odd degree and weight $w$. We wish to show that it has $\ell$-mixed weight of at most $w-2p^{n+1}$. Since the monomial has odd degree, we must have
        \[
            \epsilon_{n+1}+\epsilon_{n+2}+\cdots\geq1.
        \]
        Then, by the definition of weight, we have
        \[
            w = (2pi_1 + 2p^2i_2+\cdots) + (2p^{n+1}\epsilon_{n+1}+2p^{n+2}\epsilon_{n+2}+\cdots) \geq (2pi_1 + 2p^2i_2+\cdots) + 2p^{n+1}.
        \]
        Therefore, the $\ell$-mixed weight can be bounded as
        \[
            \sum_{j<\ell} 2p^ji_j + 2p^\ell\sum_{j\geq\ell}i_j \leq \sum_j 2p^j i_j \leq w-2p^{n+1}
        \]
        as desired.
    \end{proof}

    \begin{lemma}[{\cite[Section 5]{AL}}]\label{lem:maxdeg}
        Suppose that $x\in\mathcal A//E(n)$ is a homogeneous element of odd degree and weight $w$. Then, the degree of $x$ is at most
        \[
            -\frac{p-1}pw + 1-2p^n
        \]
        and is at least
        \[
            -w+1.
        \]
    \end{lemma}
    \begin{proof}
         Let
        \[
        P^{(i_1,i_2,\dots)}Q_{n+1}^{\epsilon_{n+1}}Q_{n+2}^{\epsilon_{n+2}}\cdots
        \]
        be a monomial of odd degree and weight $w$. We wish to minimize or maximize
        \[
            \sum_{j\geq 1} (2p^{j}-2) i_j + \sum_{j\geq n+1} (2p^{j}-1)\epsilon_j = w - \sum_{j\geq 1}2i_j - \sum_{j\geq n+1} \epsilon_j.
        \]
        Since the monomial has odd degree, we must have
        \[
            \sum_{j\geq n+1}\epsilon_j\geq1,
        \]
        so that
        \[
        w - \sum_{j\geq 1}2i_j - \sum_{j\geq n+1} \epsilon_j\leq w-1.
        \]
        Also, we have
        \[
            w = \sum_{j\geq 1}2p^j i_j + \sum_{j\geq n+1}2p^j\epsilon_{j}\geq 2p\sum_{j\geq 1}i_j + 2p^{n+1}\sum_{j\geq n+1}\epsilon_j.
        \]
        Therefore, we have
        \[
            w - \sum_{j\geq1}2i_j - \sum_{j\geq n+1}\epsilon_j \geq \frac{p-1}pw +(2p^n - 1)\sum_{j\geq n+1}\epsilon_j \geq \frac{p-1}pw + 2p^n-1
        \]
        as desired.
    \end{proof}
    
    \section{Uniqueness of \texorpdfstring{$p$}{p}-complete \textbpn}\label{sec:uniquepadic}
    In this section, we shall prove Theorem \ref{thm:intro6}. As discussed in the introduction, this implies Theorem \ref{thm:intro1}.

    Let $n\geq m\geq0$ be two nonnegative integers. Then, the Adams spectral sequence for maps from $\BP\langle n\rangle$ to $\BP\langle m\rangle$ has the form
    \[
        E_2^{s,t} = \Ext_{\mathcal A}^{s,t}(\mathcal A//E(m),\mathcal A//E(n)) \Rightarrow [\Sigma^{t-s}\BP\langle n\rangle_p^\wedge,\BP\langle m\rangle_p^\wedge].
    \]
    By change of rings, the $E_2$-page is isomorphic to
    \[
        \Ext_{\mathcal A}^{s,t}(\mathcal A//E(m),\mathcal A//E(n)) = \Ext_{E(m)}^{s,t}(\mathbb F_p,\mathcal A//E(n)).
    \]
    
    \begin{definition}\label{def:vanish}
        We say that a pair of nonnegative integers $(m,n)$ (or a triple $(p,m,n)$ if we want to make the prime $p$ explicit) with $m\leq n$ satisfies the \emph{vanishing line hypothesis} if it satisfies the conclusion of Theorem \ref{thm:intro6}, i.e. there is a positive real number $a>0$ depending only on $p,m$, and $n$ such that the group
        \[
            \Ext_{E(m)}^{s,t}(\mathbb F_p,\mathcal A//E(n))
        \]
        is zero if $t-s$ is odd and $a s + (t-s) > 1-2p^{n+1}$. Informally, this says that for odd degree classes in the Adams chart, there is a vanishing line that has negative slope and meets the $x$-axis at $(1-2p^{n+1},0)$.

        We say that $(m,n)$ satisfies the \emph{weak vanishing line hypothesis} if the same holds with $a=0$, i.e. there is a vertical vanishing line at $t-s=1-2p^{n+1}$ in the Adams chart.
    \end{definition}
    
    \begin{question}\label{q:vanish}
        Does every pair $(m,n)$ satisfy the vanishing line hypothesis?
    \end{question}
    In this section, we shall give a positive answer to this question for most cases.

    \begin{remark}\label{rmk:conv}
        The Adams spectral sequence
        \[
        E_2^{s,t} = \Ext_{\mathcal A}^{s,t}(\mathcal A//E(m),\mathcal A//E(n)) \Rightarrow [\Sigma^{t-s}\BP\langle n\rangle_p^\wedge,\BP\langle m\rangle_p^\wedge].
        \]
        conditionally converges in the sense of Boardman \cite{CCSS}. It will be clear from the computations of this section that each group $E_2^{s,t}$ is finite. Therefore, the spectral sequence converges strongly.
    \end{remark}
    
    \subsection{Ext over exterior algebra}\label{ssec:Ext}
    Let us first review the computation of Ext groups of graded modules over $E(m)$.
    \begin{definition}\label{def:bidegree}
        Suppose that $M$ is a graded $E(m)$-module. The degrees of elements of $M$ will be called the \emph{algebraic degree}. In the usual notation $\Ext^{s,t}$, the degrees $s$ and $t$ will be called the \emph{Adams degree} and the \emph{algebraic degree}. The term \emph{topological degree} will always mean $t-s$, the algebraic degree minus the Adams degree.
    \end{definition}
    \begin{definition}\label{def:resolution}
        Suppose that $m$ is a nonnegative integer. The \emph{Koszul resolution} of $\mathbb F_p$ is a differential graded $E(m)$-module $(K(E(m)),d)$ defined as a free $E(m)$-module on symbols $u^R$ where $R=(r_0,r_1,\dots,r_m)$ varies over all sequences of nonnegative integers of length $m+1$. The differential $d$ is defined by
        \[
            d(u^R) = \sum_{i=0}^m Q_i u^{R-\mathbf f_i}.
        \]
        We say that the symbol $u^R$ has algebraic degree
        \[
            -\sum_{i=0}^m r_i (2p^i-1)
        \]
        and has Adams degree
        \[
            -\sum_{i=0}^m r_i
        \]
        so that $d$ preserves the algebraic degree and increases the Adams degree by $1$.
    \end{definition}
    \begin{construction}\label{cons:Ext}
    If $M$ is an $E(m)$-module, then dualizing the above resolution, we can see that the groups $\Ext^{s,t}_{E(m)}(\mathbb F_p,M)$ can be computed as the cohomology of the complex $(M[v_0,\dots,v_n],d)$ where
    \[
        M[v_0,\dots,v_m] := \mathbb F_p[v_0,\dots,v_m]\otimes_{\mathbb F_p}M
    \]
    and $d$ is defined as
    \[
        d x = \sum_{i=0}^m v_i(Q_i x)
    \]
    for $x\in M$, extended $\mathbb F_p[v_0,\dots,v_m]$-linearly. Here, $v_0^{r_0}\cdots v_m^{r_m}$ is the dual of $u^{(r_0,\dots,r_m)}$. The bidegree of $v_i$ is $(1,2p^i-1)$, i.e. its Adams degree is $1$ and its algebraic degree is $2p^i-1$. Then, $d$ preserves the algebraic degree and increases the Adams degree by $1$. After taking the cohomology with respect to $d$, this notion of bidegree agrees with that of Definition \ref{def:bidegree}.
    \end{construction}
    \begin{example}\label{ex:KD}
        For $M=\mathbb F_p$, we have
        \[
            \Ext_{E(m)}^{\ast\ast}(\mathbb F_p,\mathbb F_p) = \mathbb F_p[v_0,\dots,v_m]
        \]
        by above. Also, by Koszul duality of exterior and polynomial algebras, the ring structure on $\Ext_{E(m)}(\mathbb F_p,\mathbb F_p)$, induced by the Yoneda product, is the polynomial ring structure. For any $E(m)$-module $M$, the group $\Ext_{E(m)}(\mathbb F_p,M)$ is a module over $\Ext_{E(m)}(\mathbb F_p,\mathbb F_p)$ by the Yoneda product, and is given by the obvious $\mathbb F_p[v_0,\dots,v_m]$-module structure on $M[v_0,\dots,v_m]$ in the description of Ext groups in Construction \ref{cons:Ext}.
    \end{example}
    \subsection{Vanishing lines}
    In this section, we shall prove the existence of vanishing lines (Question \ref{q:vanish}) for the groups
    \[
        \Ext^{s,t}_{E(m)}(\mathbb F_p,\mathcal A//E(n)).
    \]

    Recall from Construction \ref{cons:Ext} that the above Ext groups can be computed as the cohomology groups of
    \[
        \left((\mathcal A//E(n))[v_0,\dots,v_m], d=\sum_{i=0}^m v_i Q_i\right).
    \]
    We can extend the weight grading and $\ell$-mixed weight gradings of $\mathcal A//E(n)$ (Definitions \ref{def:weight} and \ref{def:ellweight}) to $(\mathcal A//E(n))[v_0,\dots,v_m]$ by declaring that each $v_i$ has weight $0$ and $\ell$-mixed weight $0$ for all $\ell$.

    \begin{lemma}
        The group $(\mathcal A//E(n))[v_0,\dots,v_m]$ is trigraded by Adams degree, algebraic degree, and weight. The differential $d$ preserves the algebraic degree and the weight, and $d$ increases the Adams degree by $1$.
    \end{lemma}
    \begin{proof}
        The only nontrivial assertion is that the three gradings are compatible with each other. This follows from the fact that each graded piece of each grading is generated by monomials.
    \end{proof}

    \begin{convention}
        From this point, all elements of $(\mathcal A//E(n))[v_0,\dots,v_m]$ that we consider are assumed to be homogeneous in the tridegree, i.e. Adams degree, algebraic degree, and weight.
    \end{convention}
    \begin{construction}\label{cons:normal}
    Let $0\neq\chi\in\Ext^{\ast\ast}_{E(m)}(\mathbb F_p,\mathcal A//E(n))$ be a class of odd topological degree and weight $w$. Suppose that $x\in (\mathcal A//E(n))[v_0,\dots,v_m]$ represents $\chi$. Then, let $x_0=x$, and we inductively define a sequence $x_1,\dots,x_{m+1}\in(\mathcal A//E(n))[v_0,\dots,v_m]$ and a sequence of integers $r_0,\dots,r_m$ by
    \[
        x_i = v_i^{r_i}x_{i+1} \pmod{v_i^{r_i+1}}, \qquad 0\neq x_{i+1}\in(\mathcal A//E(n))[v_{i+1},\dots,v_m].
    \]
    Note that since the $v_i$'s have even topological degrees, each $x_i$ has odd topological degree.
    
    Let us say that $x$ is a \emph{normal representation} of $\chi$ if $(r_0,\dots,r_m)$ is lexicographically minimal among all representatives of $\chi$. We shall say that $(x_0,\dots,x_{m+1})$ and $(r_0,\dots,r_m)$ are the \emph{associated sequences}.
    \end{construction}

    \begin{proposition}\label{prop:mainineq}
        In the situation of Construction \ref{cons:normal} where $x$ is a normal representation, we have $r_m=0$ and
        \[
            r_i \leq \frac{w-2p^{n+1}}{2p^n(p-1)}
        \]
        for all $0\leq i<m$.
    \end{proposition}

    The proof of the proposition is a combination of the following lemmas. Let us write
    \[
        d^i = v_iQ_i + \cdots+ v_mQ_m
    \]
    for the differential for the quotient complex $(\mathcal A//E(n))[v_i,\dots,v_m]$. Note that $d=d^0$.
    \begin{lemma}
        In the situation of Construction \ref{cons:normal} where $x$ is a normal representation with associated sequences $(x_0,\dots,x_{m+1})$ and $(r_0,\dots,r_m)$, $x_i$ is a $d^i$-cycle but not a $d^i$-boundary for all $0\leq i\leq m$.
    \end{lemma}
    \begin{proof}
        Let us induct on $i$. For $i=0$, this follows from the assumptions.
        
        Suppose $i>0$. Since
        \[
            x_{i-1} = v_{i-1}^{r_{i-1}}x_i \pmod{v_{i-1}^{r_{i-1}+1}}
        \]
        is a cycle with respect to $d^{i-1} = v_{i-1}Q_{i-1}+d^i$, it follows that $x_i$ is a $d^i$-cycle.
        
        Suppose that $x_i = d^i y$ for some $y$. Consider the cycle
        \[
            x' := x - v_0^{r_0}\cdots v_{i-1}^{r_{i-1}}dy = x - v_0^{r_0}\cdots v_{i-1}^{r_{i-1}}(v_0Q_0y + \cdots+v_{i-1}Q_{i-1}y + d^iy)
        \]
        and let $(x_0',x_1',\dots,x_m',x_{m+1}')$ and $(r_0',\dots,r_m')$ be the sequences defined as before for $x'$ instead of $x$.
        
        From the definition of $x'$, it must be a multiple of $v_0^{r_0}$ since $x_0$ is. However, by the lexicographical maximality of $(r_0,\dots,r_m)$, the cycle $x'$ cannot be a multiple of $v_0^{r_0+1}$. This implies that $r_0'=r_0$, and looking at the definition of $x'$ modulo $v_0^{r_0+1}$, we have
        \[
            x_1' = x_1 - v_1^{r_1}\cdots v_{i-1}^{r_{i-1}}(v_1Q_1+\cdots+v_{i-1}Q_{i-1}y + d^iy).
        \]
        Repeating the same argument $i$ times, we eventually obtain
        \[
            x_i' = x_i - d^iy = 0
        \]
        which is a contradiction.
    \end{proof}

    \begin{lemma}
        Suppose that $y\in(\mathcal A//E(n))[v_i,\dots,v_m]$ is an element of odd degree and weight $w$. Also, suppose that it is a $d^i$-cycle but not a $d^i$-boundary. If $y$ is divisible by $v_i^r$, then we have
        \[
            r\leq\frac{w-2p^{n+1}}{2p^n(p-1)}
        \]
        if $i<m$ and $r=0$ if $i=m$.
    \end{lemma}
    \begin{proof}
        If $i=m$, then $y$ is of the form $v_m^{r_m}y'$ for $y'\in\mathcal A//E(n)$ by homogeneity. Since $y$ is a cycle with respect to $d^m=v_mQ_m$, we have $Q_my'=0$. By Proposition \ref{prop:exact}, there is some $z$ such that $y'=Q_mz$. If $r>0$, then $r_m>0$ and we have $v_m^{r_m}y'=d^m(v_m^{r_m-1}z)$, which is a contradiction. Therefore, we have $r=0$.
    
        Now suppose that $i<m$. Let us write
        \[
            y = v_i^r y_r + v_i^{r+1}y_{r+1}+\cdots+ v_i^ky_k
        \]
        for some $r\leq k$ and $y_r,\dots,y_k\in (\mathcal A//E(n))[v_{i+1},\dots,v_m]$. Since $y$ is a cycle with respect to $d^i = v_iQ_i + d^{i+1}$, we can see that $Q_iy_k=0$. By Proposition \ref{prop:exact}, we have $y_k = Q_i z_k$ for some $z_k$.

        If $k>r$, we can replace $y$ with
        \[
        y-v_i^{k-1}d^iz_k = v_i^ry_r + \cdots+ v_i^{k-1}(y_{k-1}-d^{i+1}z_k)
        \]
        so that we have lowered $k$ by $1$ while $y$ is still a multiple of $v_i^r$. Repeating the argument, we may assume that
        \[
            y = v_i^r y_r
        \]
        for some $y_r\in(\mathcal A//E(n))[v_{i+1},\dots,v_m]$.

        Now, we shall inductively define $y_{r-1},\dots,y_0\in(\mathcal A//E(n))[v_{i+1},\dots,v_m]$ so that $y$ is $d^i$-homologous to $v_i^{j}y_j$ for all $0\leq j\leq r$. Inductively, if $y$ is $d^i$-homologous to $v_i^jy_j$, then since $v_i^jy_j$ is a cycle, we have $Q_i y_j=0$, so that $y_j = Q_i z_j$ for some $z_j$ by Proposition \ref{prop:exact}. Then, $y$ is $d^i$-homologous to
        \[
            v_i^j y_j - v_i^{j-1}d^iz_j = -v_i^{j-1}(d^{i+1}z_j)
        \]
        and we can define $y_{j-1} := -d^{i+1}z_j$.

        In the previous paragraph, if we have
        \[
            y_j \in F^{n,n-i}_s[v_{i+1},\dots,v_m] := F^{n,n-i}_s\otimes_{\mathbb F_p}\mathbb F_p[v_{i+1},\dots,v_m]
        \]
        (Definition \ref{def:ellweight}), then by Lemma \ref{lem:ellweightineq} applied to $(k,\ell)=(i,n-i)$, we have
        \[
            z_j \in F^{n,n-i}_{s+2p^n}[v_{i+1},\dots,v_m],
        \]
        and by the same Lemma, we have
        \[
            y_{j-1} \in F^{n,n-i}_{s+2p^n-2p^{n+1}}[v_{i+1},\dots,v_m].
        \]
        Therefore, since
        \[
            y_r\in F^{n,n-i}_{w-2p^{n+1}}[v_{i+1},\dots,v_m]
        \]
        by Lemma \ref{lem:ellweightub}, we must have
        \[
            y_0 \in F^{n,n-i}_{w-2p^{n+1} + r(2p^n-2p^{n+1})}[v_{i+1},\dots,v_m].
        \]
        Since $y$ is not a $d^i$-boundary, it follows that $y_0$ is nonzero. However, since the $\ell$-mixed weight is always nonnegative, we have
        \[
w-2p^{n+1} + r(2p^n-2p^{n+1})\geq0
        \]  
        which implies the desired statement.
    \end{proof}
    \begin{proof}[Proof of Proposition \ref{prop:mainineq}]
        It follows immediately from the previous two lemmas.
    \end{proof}

    With Proposition \ref{prop:mainineq}, we can answer Question \ref{q:vanish}.
    \begin{theorem}[Theorem \ref{thm:intro6}]\label{thm:vanish}
        Let $m\leq n$ be nonnegative integers. Suppose that at least one of the following conditions is true:
        \begin{enumerate}
        \item $p$ is odd,
        \item $m\leq n-1$,
        \item $n\leq 2$.
        \end{enumerate}
        Then, $(m,n)$ satisfies the vanishing hypothesis (Definition \ref{def:vanish}). Also, the triple $(p,m,n)=(2,3,3)$ satisfies the weak vanishing hypothesis.
    \end{theorem}
    \begin{proof}
        Let $0\neq \chi\in\Ext_{E(m)}^{\ast\ast}(\mathbb F_p,\mathcal A//E(n))$ be a nonzero class of odd topological degree and weight $w$, and let $x$ be a normal representation (Construction \ref{cons:normal}) with associated sequences $(x_0,\dots,x_{m+1})$ and $(r_0,\dots,r_m)$. Note that we have $r_m=0$ by Proposition \ref{prop:mainineq}. Thus, if $m=0$, the statement vacuously follows since there is no odd Ext class in positive Adams degree and all odd degree elements of $\mathcal A//E(n)$ have degrees at most $1-2p^{n+1}$.
        
        Now suppose that $m>0$. Then, since
        \[
            x=v_0^{r_0}\cdots v_{m-1}^{r_{m-1}}x_{m+1}+\cdots
        \]
        with $x_{m+1}\in\mathcal A//E(n)$, we have
        \[
            \deg(x) \leq \left(\sum_{j=0}^{m-1} (2p^j-2)r_j\right)-\frac{p-1}pw+1-2p^n
        \]
        by Lemma \ref{lem:maxdeg}, where $\deg$ denotes the topological degree. Rearranging, we have
        \[
            w - 2p^{n+1} \leq \frac{p}{p-1}\left(1-2p^{n+1}-\deg(x) +\sum_{j=0}^{m-1} (2p^j-2)r_j\right).
        \]

        By Proposition \ref{prop:mainineq}, we have
        \begin{align*}
            &\sum_{j=0}^{m-1}(2p^{j+1}-2p^j)r_j \\
            &\leq \left(p^{-n}+\cdots+p^{-n+m-1}\right)(w-2p^{n+1})\\
            &\leq \frac{p^{-n+m}}{p-1} (w-2p^{n+1})\\
            &\leq \frac{p^{-n+m+1}}{(p-1)^2}\left(1-2p^{n+1}-\deg(x) +\sum_{j=0}^{m-1} (2p^j-2)r_j\right).
        \end{align*}
        Suppose that $p$ is odd. Then, we have
        \[
            2p^{j+1}-2p^j > \frac{p(2p^j-2)}{(p-1)^2} \geq \frac{p^{-n+m+1}(2p^j-2)}{(p-1)^2}
        \]
        so that we have $a(r_0+\cdots+r_{m-1})\leq 1-2p^{n+1}-\deg(x)$
        for some $a>0$, which is the desired statement. The same proof works for the case $p=2$ and $m=1$.
        
        Similarly, if $m\leq n-1$, then we have
        \[
            2p^{j+1}-2p^j > \frac{2p^j-2}{(p-1)^2} \geq \frac{p^{-n+m+1}(2p^j-2)}{(p-1)^2}
        \] 
        and the same conclusion follows.

        Finally, let us handle the case where $p=2$ and $2\leq n=m\leq 3$. If $n=m=2$, then by Proposition \ref{prop:mainineq}, we have
        \[
            8r_0,8r_1 \leq w-2p^{n+1} \leq 2(1-2p^{n+1}-\deg(x) + 2r_1),
        \]
        from which the theorem easily follows. For $n=m=3$, we have
        \[
            16r_0,16r_1,16r_2\leq 2(1-2p^{n+1}-\deg(x) + 2r_1 + 6r_2).
        \]
        Then, we have
        \[
            16r_1 + 48r_2 \leq 8(1-2p^{n+1}-\deg(x) +2r_1+6r_2)
        \]
        so that
        \[
            \deg(x)\leq 1-2p^{n+1}
        \]
        as desired.
    \end{proof}

    \begin{remark}
        We are not suggesting that Question \ref{q:vanish} has a negative answer when $p=2$. It is simply that the inequalities we have here are not sufficient for those cases. The author expects that Theorem \ref{thm:vanish} is true for all $p$ and $m=n$.
    \end{remark}

    \begin{remark}\label{rmk:AL}
        Let us briefly discuss the proof of \cite{AL} of Theorem \ref{thm:vanish}. For simplicity, let us consider the case of $\BP\langle 2\rangle$ at $p=2$.

        Given an odd degree cycle, let us consider the normal representation (Construction \ref{cons:normal})
        \[
            x=v_0^{r_0}v_1^{r_1}x_{3} + \cdots.
        \]
        In other words, the rest of the terms are multiples of $v_0^i v_1^j v_2^k$ where $(i,j,k)$ is lexicographically larger than $(r_0,r_1,0)$.
        Then, the argument of \cite{AL} claims that there is another representative, modulo cycles and lexicographically large terms (larger than $(r_0,r_1,0)$), which is a sum of terms of the form
        \[
            v_1^{r_0-i}v_2^{r_1+i} y
        \]
        for $i\geq0$. However, computer-assisted calculation shows that the cycle $v_0^3 v_1^2  P^{(8,0,0,4)}Q_3$ does not satisfy this, and we can check that it meets a dead-end while going through the zig-zag algorithm of \cite[Section 5]{AL}.
        %
        %

    \end{remark}

    \section{Recovering \textp-local homotopy types from \textp-complete homotopy types}\label{sec:AParg}
    In this section, we study how to lift an equivalence of $p$-complete spectra to an equivalence of $p$-local spectra.
    \subsection{Lifting \textp-complete maps to \textp-local maps}
    \begin{definition}
        Let $X,Y$ be $p$-local spectra of finite type. We say that the pair $(X,Y)$ is \emph{preAP} if the natural map
        \[
            [X_p^\wedge,Y_p^\wedge]\otimes_{\mathbb Z}\mathbb Q \to [X_p^\wedge,\tau_{\leq k}Y_p^\wedge]\otimes_{\mathbb Z}\mathbb Q
        \]
        is surjective for all $k\in\mathbb Z$. Note that
        \[
            [X_p^\wedge,\tau_{\leq k}Y_p^\wedge]\otimes_{\mathbb Z}\mathbb Q=[\tau_{\leq k}X_p^\wedge,\tau_{\leq k}Y_p^\wedge]\otimes_{\mathbb Z}\mathbb Q = \prod_{c\leq k} \Hom_{\mathbb Q_p}(\pi_c(X_p^\wedge)\otimes_{\mathbb Z}\mathbb Q,\pi_c(Y_p^\wedge)\otimes_{\mathbb Z}\mathbb Q)
        \]
        is the endomorphism group of $p$-adic rational homotopy groups up to degree $k$. We say that the pair $(X,Y)$ is \emph{AP} if $(X,\Sigma^jY)$ is preAP for all $j\in\mathbb Z$.
    \end{definition}
    \begin{example}\label{ex:BP1AP}
        Consider the case $X=Y=\ell$ where $\ell$ is the connective Adams summand of the $p$-local connective complex $K$-theory, which is a form of $\BP\langle 1\rangle$. We would like to show that $(\ell,\ell)$ is preAP.
        
        For each $r\in1+p\mathbb Z_p$, there is an Adams operation $\Psi^r$, which is an $\mathbb E_\infty$-ring automorphism of $\ell$ and acts on homotopy groups by sending $v_1$ to $rv_1$.

        Let $k$ be a fixed nonnegative integer, and let $r_0,r_1,\dots,r_k\in 1+p\mathbb Z_p$ be any distinct elements. Then, the image of $\Psi^{r_i}$ under the map
        \[
            [\ell_p^\wedge,\ell_p^\wedge]\to\prod_{c\leq 2k} \Hom_{\mathbb Q_p}(\pi_c(\ell_p^\wedge)\otimes_{\mathbb Z}\mathbb Q,\pi_c(\ell_p^\wedge)\otimes_{\mathbb Z}\mathbb Q) = \mathbb Q_p^{k+1}
        \]
        is the vector
        \[
            \begin{bmatrix}
                1\\r_i\\\vdots\\r_i^k
            \end{bmatrix}.
        \]
        Therefore, the images of the $\Psi^{r_i}$'s form a Vandermonde matrix, which is known to be nonsingular. Since $k$ was arbitrary, this proves that $(\ell,\ell)$ is preAP. Later, we shall prove that the pair is actually AP (Theorem \ref{thm:BPnAP}).
    \end{example}
    \begin{lemma}\label{lem:plocallift}
        Let $Y$ be a $p$-local spectrum. Given a map $f:X\to Y_p^\wedge$, it lifts to a map $\tilde{f}:X\to Y$ making the diagram
        \[
            \begin{tikzcd}
                X\ar[r,"\tilde{f}"]\ar[dr,"f",swap]& Y\ar[d]\\
                &Y_p^\wedge
            \end{tikzcd}
        \]
        commutative, if and only if there is a lift after taking $\pi_\ast$, i.e. there is a map of graded abelian groups $F_\ast$ making the diagram
        \[
            \begin{tikzcd}
                \pi_\ast (X)\ar[r,"F_\ast"]\ar[dr,swap,"\pi_\ast (f)"]& \pi_\ast(Y)\ar[d]\\
                &\pi_\ast (Y_p^\wedge)
            \end{tikzcd}
        \]
        commutative.
    \end{lemma}
    \begin{proof}
        The only if part is obvious. Let us prove the if part. Since the arithmetic fracture square
        \[
            \begin{tikzcd}
                Y\ar[r]\ar[d]&Y_p^\wedge\ar[d]\\
                Y\otimes\mathbb Q\ar[r] &Y_p^\wedge\otimes\mathbb Q
            \end{tikzcd}
        \]
        is a pullback diagram, we can see that
        \[
            \cof(Y\to Y_p^\wedge)\simeq \cof(Y\otimes\mathbb Q\to Y_p^\wedge\otimes\mathbb Q)
        \]
        is a rational spectrum. Then, whether the composite
        \[
            X\to Y_p^\wedge \to \cof(Y\to Y_p^\wedge)
        \]
        is zero or not can be checked at the level of homotopy groups, since maps to rational spectra are determined by the map of homotopy groups. Since the map $\pi_\ast(X)\to\pi_\ast(Y_p^\wedge)$ factors through $\pi_\ast(Y)$, we see that the above composite is zero and that $f$ lifts to a map $\tilde{f}:X\to Y$.
    \end{proof}
    \begin{proposition}[cf. \cite{AP}]\label{prop:AParg}
        Let $X$ and $Y$ be $p$-local spectra of finite type. If $(X,Y)$ is preAP, then for any map $f:X_p^\wedge\to Y_p^\wedge$ and a positive integer $r$, there exists $\tilde{g}:X\to Y$ such that $f/p^r \simeq \tilde{g}/p^r$ as maps from $X_p^\wedge/p^r=X/p^r$ to $Y_p^\wedge/p^r=Y/p^r$. In particular, if there is an equivalence $X_p^\wedge\simeq Y_p^\wedge$, then there is an equivalence $X\simeq Y$.
    \end{proposition}
    \begin{proof}
        By taking suspensions, let us assume that $X$ and $Y$ are connective. For each $i\in\mathbb Z$, let us fix identifications
        \begin{align*}
            \pi_i(X) &= \mathbb Z_{(p)}^{\oplus s_i}\oplus\text{(torsion)}\\
            \pi_i(Y) &= \mathbb Z_{(p)}^{\oplus t_i}\oplus\text{(torsion)}
        \end{align*}
        for integers $s_i$ and $t_i$, giving us an identification
        \[
            \pi_i(Y_p^\wedge) = \mathbb Z_p^{\oplus t_i}\oplus\text{(torsion)}.
        \]
        Then, if we ignore the torsion part, given a map $g:X\to Y_p^\wedge$, the induced map $\pi_i(g):\pi_i(X)\to \pi_i(Y_p^\wedge)$ can be represented by a $s_i\times t_i$ matrix with coefficients in $\mathbb Z_p$, which we also denote by $\pi_i(g)$. Note that by Lemma \ref{lem:plocallift}, the map $g$ lifts to a map $X\to Y$ if and only if $\pi_i(g)$ has coefficients in $\mathbb Z_{(p)}$ for all $i$, or as we shall say, \emph{has $p$-local coefficients}.

        Let us write $g_{-1}:=f$ and inductively define maps $g_0,g_1,\dots:X\to Y_p^\wedge$, so that $\pi_i(g_k)$ has $p$-local coefficients for all $i\leq k$.

        Suppose that we have defined $g_{k-1}$ for some $k\geq 0$. Since $(X,Y)$ is preAP, the map
        \[
            [X,Y_p^\wedge]\otimes_{\mathbb Z}\mathbb Q\to [X,\tau_{\leq k}Y_p^\wedge]\otimes_{\mathbb Z}\mathbb Q
        \]
        is surjective. Since the right hand side is a finite dimensional $\mathbb Q_p$-vector space and $[X,Y_p^\wedge]$ is a $\mathbb Z_p$-module, there is an integer $N$ (dependent on $k$) satisfying the following property.
        \begin{quote}
            For any sequence of matrices $A_0,\dots, A_k$ with $\mathbb Z_p$-coefficients where $A_i$ has size $s_i\times t_i$ for each $i$, there is a map $h:X\to Y_p^\wedge$ such that $\pi_i(h)=p^N A_i$ for every $i\leq k$.
        \end{quote}
        Let $B_k$ be a $s_k\times t_k$ matrix with $\mathbb Z_p$-coefficients such that
        \[
            \pi_k(g_{k-1})+p^{k+r+1+N}B_k
        \]
        has $p$-local coefficients. Such a matrix exists since $\mathbb Z_p/\mathbb Z_{(p)}$ is a rational vector space. Then, there is a map $h_k:X\to Y_p^\wedge$ such that $\pi_i(h_k)=0$ for $i<k$ and $\pi_k(h_k)=p^NB_k$. If we define $g_k:=g_{k-1}+p^{k+r+1}h_k$, then $\pi_i(g_k)$ has $p$-local coefficients for all $i\leq k$.

        Let us define $g$ to be the limit of $g_k$, or more precisely, as
        \[
            f+p^{r+1}h_0 + p^{r+2}h_1+\cdots,
        \]
        which makes sense, not necessarily in a unique way, since the function spectrum $F(X,Y_p^\wedge)$ is $p$-complete. Then, by construction, $\pi_i(g)$ has $p$-local coefficients for all $i$, so that $g$ lifts to a map $\tilde{g}:X\to Y$. Also, we have $\tilde{g}/p^r =f/p^r$ since $p^{r+1}$ is zero in $[\mathbb S/p^r,\mathbb S/p^r]$.

        The last assertion of the proposition follows from the fact that an equivalence of $p$-local spectra can be tested modulo $p^r$ for any $r>0$.
    \end{proof}

    \begin{example}\label{ex:Zplusell}
        Let us discuss an example of a pair that is not preAP. Consider the pair
        \[
        (\mathbb Z_{(p)},\ell)=(\BP\langle0\rangle,\BP\langle1\rangle).
        \]
        Then, this pair is not preAP because there is no map $\mathbb Z_p\to\ell_p^\wedge$ that is nonzero on $\pi_0$. In fact, any such map is zero since
        \[
            \Map(\mathbb Z_p,\ell_p^\wedge) = \Map(\mathbb Z_p,\ell_p^\wedge[v_1^{-1}])
        \]
        and $\ell_p^\wedge[v_1^{-1}]$ is $K(1)$-local while $\mathbb Z_p$ is $K(1)$-acyclic. It also follows that the pair
        \[
            (\mathbb Z_{(p)}\oplus\ell,\mathbb Z_{(p)}\oplus\ell)
        \]
        is not preAP.

        Furthermore, there exists a $p$-local spectrum $X$ of finite type that is equivalent to $\mathbb Z_{(p)}\oplus\ell$ after $p$-completion but not before. From the arithmetic fracture square, there is a long exact sequence
        \[
            \cdots\to[\mathbb Z_{(p)},\ell_p^\wedge\oplus(\ell\otimes\mathbb Q)] \to [\mathbb Z_{(p)},\ell_p^\wedge\otimes\mathbb Q]\to [\mathbb Z_{(p)},\Sigma\ell]\to\cdots
        \]
        and by the previous observation that $[\mathbb Z_{(p)},\ell_p^\wedge]=0$, there is an injection
        \[
            \mathbb Q_p/\mathbb Q\to[\mathbb Z_{(p)},\Sigma\ell]
        \]
        corresponding to maps that become zero after $p$-completion. Then, take any such nonzero map $f:\mathbb Z_{(p)}\to\Sigma\ell$ and let $X$ be the fiber of $f$. Since $f$ becomes zero after $p$-completion, such a nullhomotopy induces a splitting $X_p^\wedge = \mathbb Z_p\oplus \ell_p^\wedge$.

        Suppose that there is a split injection $i:\mathbb Z_{(p)}\to X$. Then, after $p$-completion, we get a split injection $i_p:\mathbb Z_p\to X_p^\wedge=\mathbb Z_p\oplus\ell_p^\wedge$, but since $[\mathbb Z_p, \ell_p^\wedge]=0$, $i_p$ must map $\mathbb Z_p$ isomorphically onto the $\mathbb Z_p$ summand. In other words, the composite
        \[
            \mathbb Z_p \xrightarrow{i_p} X_p\to \mathbb Z_p 
        \]
        is an equivalence, where the second arrow comes from the definition of $X$ as the fiber of $f$. This implies that the composite
        \[
            \mathbb Z_{(p)}\xrightarrow{i} X\to\mathbb Z_{(p)}
        \]
        is an equivalence, which contradicts that the map $X\to \mathbb Z_{(p)}$ does not split.
    \end{example}
    \subsection{AP pairs of spectra}
    In this section, we shall prove various properties of AP pairs of spectra. We shall work with the following equivalent definition in terms of $p$-local spectra instead of $p$-complete spectra.
    \begin{lemma}
        Let $X,Y$ be $p$-local spectra of finite type. Then, the pair $(X,Y)$ is preAP if and only if
        \[
            [X,Y]\otimes_{\mathbb Z}\mathbb Q\to [X,\tau_{\leq k}Y]\otimes_{\mathbb Z}\mathbb Q
        \]
        is surjective for all $k$.
    \end{lemma}
    \begin{proof}
        For the if part, consider the square
        \[
            \begin{tikzcd}[]
                [X,Y]\otimes_{\mathbb Z}\mathbb Q\ar[r]\ar[d]&[]
                [X,\tau_{\leq k}Y]\otimes_{\mathbb Z}\mathbb Q\ar[d]\\
                [][X_p^\wedge,Y_p^\wedge]\otimes_{\mathbb Z}\mathbb Q\ar[r]&[][X_p^\wedge,\tau_{\leq k}Y_p^\wedge]\otimes_{\mathbb Z}\mathbb Q.
            \end{tikzcd}
        \]
        Since the lower horizontal arrow is a map of $\mathbb Q_p$-vector spaces, we can basechange the top row along $\mathbb Q\to\mathbb Q_p$. Then, the right vertical arrow becomes an isomorphism, and if the upper horizontal arrow is a surjection, so is the lower horizontal arrow.

        For the only if part, consider a map in
        \[
            f\in[X,\tau_{\leq k}Y]\otimes_{\mathbb Z}\mathbb Q = \prod_{c\leq k} \Hom_{\mathbb Q}(\pi_c(X)\otimes_{\mathbb Z}\mathbb Q,\pi_c(Y)\otimes_{\mathbb Z}\mathbb Q).
        \]
        Since $(X,Y)$ is preAP, after multiplying a large power of $p$ to $f$ if necessary, there is a map $X\to Y_p^\wedge$ such that the map induces $f$ on homotopy groups of degree $\leq k$. Then, from the proof of Proposition \ref{prop:AParg}, there is a map $X\to Y_p^\wedge$ such that it induces $f$ on homotopy groups of degree $\leq k$ and it lifts to a map $X\to Y$. This can be done by modifying the proof to set $h_k=0$ whenever $\pi_k(g_{k-1})$ already has $p$-local coefficients.
    \end{proof}

    \begin{proposition}\label{prop:tensorfin}
        Let $A$ be a finite spectrum such that $A\otimes\mathbb Q\neq 0$. Let $X, Y$ be $p$-local spectra of finite type.
        \begin{enumerate}
            \item The pair $(X,Y)$ is AP if and only if $(X\otimes A, Y)$ is AP.
            \item The pair $(X,Y)$ is AP if and only if $(X,Y\otimes A)$ is AP.
        \end{enumerate}
    \end{proposition}
    \begin{proof}
        \begin{enumerate}[leftmargin=20pt]
            \item We have natural isomorphisms
            \[
                [X\otimes A, Y] \otimes_{\mathbb Z}\mathbb Q
                =[A,F(X,Y)]\otimes_{\mathbb Z}\mathbb Q
                =[A\otimes\mathbb Q, F(X,Y)\otimes \mathbb Q]
            \]
            where we use that $A$ is finite for the second isomorphism. Since $A\otimes{\mathbb Q}$ is free over $\mathbb Q$, we may as well assume that $A$ is a direct sum of spheres. In that case, the statement follows from the definitions.
            \item We first prove the only if part. It is enough to prove that $(X,Y\otimes A)$ is preAP since we can repeat the argument after replacing $Y$ with suspensions of $Y$.
            
            By (a) applied to the pair $(X\otimes DA, Y)$ where $DA$ is the Spanier-Whitehead dual of $A$, we have that
            \[
                [X,Y\otimes A]\otimes_{\mathbb Z}\mathbb Q\to [X, (\tau_{\leq k} Y)\otimes A]\otimes_{\mathbb Z}\mathbb Q
            \]
            is surjective for any integer $k$. If $A$ is $c$-connective for some $c\in\mathbb Z$, then the $k$-th Postnikov truncation of $Y\otimes A$ factors as
            \[
                Y\otimes A\to (\tau_{\leq k-c} Y)\otimes A\to \tau_{\leq k}(Y\otimes A).
            \]
            Therefore, to show that
            \[
                [X,Y\otimes A]\otimes_{\mathbb Z}\mathbb Q\to[X,\tau_{\leq k}(Y\otimes A)]\otimes_{\mathbb Z}\mathbb Q
            \]
            is surjective, it is enough to show that
            \[
                [X,(\tau_{\leq k-c}Y)\otimes A]\otimes_{\mathbb Z}\mathbb Q\to [X,\tau_{\leq k}(Y\otimes A)]\otimes_{\mathbb Z}\mathbb Q
            \]
            is surjective. However, since $(\tau_{\leq k-c}Y)\otimes A$ and $\tau_{\leq k}(Y\otimes A)$ are bounded and $X$ is of finite type, we have
            \begin{align*}
                [X,(\tau_{\leq k-c}Y)\otimes A]\otimes_{\mathbb Z}\mathbb Q &= [X\otimes\mathbb Q, (\tau_{\leq k-c}Y)\otimes A\otimes\mathbb Q]\\
                [X,\tau_{\leq k}(Y\otimes A)]\otimes_{\mathbb Z}\mathbb Q&=[X\otimes\mathbb Q,\tau_{\leq k}(Y\otimes A)\otimes\mathbb Q]
            \end{align*}
            and since $X\otimes\mathbb Q$ is free, the desired statement follows from the fact that
            \[
                (\tau_{\leq k-c}Y)\otimes A\to \tau_{\leq k}(Y\otimes A)
            \]
            is surjective on rational homotopy groups.

            Next, let us prove the if part. Let $d$ be the dimension of the highest cell of $A$. Then, we have a factorization
            \[
                Y\otimes A\to \tau_{\leq k+d}(Y\otimes A)\otimes\mathbb Q\to (\tau_{\leq k}Y)\otimes A\otimes\mathbb Q
            \]
            since $(\tau_{\leq k}Y)\otimes A$ is rationally $(k+d)$-truncated. If $(X, Y\otimes A)$ is AP, then the map
            \[
                [X,Y\otimes A]\otimes_{\mathbb Z}\mathbb Q\to [X,\tau_{\leq k+d}(Y\otimes A)]\otimes_{\mathbb Z}\mathbb Q
            \]
            is surjective.
            Also, since $\tau_{\leq k+d}(Y\otimes A)$ and $(\tau_{\leq k}Y)\otimes A$ are bounded and $X$ is of finite type, we have
            \begin{align*}
                [X,\tau_{\leq k+d}(Y\otimes A)]\otimes_{\mathbb Z}\mathbb Q&=[X\otimes\mathbb Q,\tau_{\leq k+d}(Y\otimes A)\otimes\mathbb Q]\\
                [X,(\tau_{\leq k}Y)\otimes A]\otimes_{\mathbb Z}\mathbb Q&=[X\otimes\mathbb Q,(\tau_{\leq k}Y)\otimes A\otimes\mathbb Q]
            \end{align*}
            and the map
            \[
                [X,\tau_{\leq k+d}(Y\otimes A)]\otimes_{\mathbb Z}\mathbb Q\to [X,(\tau_{\leq k}Y)\otimes A]\otimes_{\mathbb Z}\mathbb Q
            \]
            induced by
            \[
                \tau_{\leq k+d}(Y\otimes A)\otimes\mathbb Q\to (\tau_{\leq k}Y)\otimes A\otimes\mathbb Q,
            \]
            which is a surjection on homotopy groups, is a surjection since $X\otimes\mathbb Q$ is free over $\mathbb Q$.

            Therefore, we have proved that
            \[
                [X,Y\otimes A]\to [X,(\tau_{\leq k}Y)\otimes A]
            \]
            is surjective for all $k$. Taking the Spanier-Whitehead dual of $A$, we have that $(X\otimes DA, Y)$ is preAP. By replacing $Y$ with suspensions of $Y$, we also have that $(X\otimes DA, Y)$ is AP. By (a), this implies that $(X,Y)$ is AP.\qedhere
        \end{enumerate}
    \end{proof}

    \begin{proposition}\label{prop:tmfBP2}
        At any prime $p$, the pair $(\tmf_{(p)},\tmf_{(p)})$ is AP if and only if $(\BP\langle 2\rangle,\BP\langle 2\rangle)$ is AP.
    \end{proposition}
    \begin{proof}
        By \cite[Theorem 1.2]{Htmf}, at $p=2$ or $3$, there is a finite spectrum $A$ such that $\tmf_{(p)}\otimes A$ is a sum of suspensions of $\BP\langle 2\rangle$. Therefore, the statement follows from Proposition \ref{prop:tensorfin}. For $p\geq 5$, we shall later prove that $\tmf_{(p)}$ is a direct sum of suspensions of $\BP\langle 2\rangle$'s (Corollary \ref{cor:tmfsplit}).
    \end{proof}

    \begin{remark}
        The hypothesis that $Y$ is of finite type is not used in the proof of Proposition \ref{prop:tensorfin}. However, in the cases of our interest, $Y$ will always be of finite type.
    \end{remark}

    The following statement is about the uniqueness of $p$-local lifts of maps in Lemma \ref{lem:plocallift}.
    \begin{proposition}\label{prop:fracture}
        Let $X$ and $Y$ be $p$-local spectra of finite type such that $(X,Y)$ is preAP. Then, the natural map
        \[
            [X,Y_p^\wedge]\oplus[X,Y\otimes\mathbb Q] \to [X,Y_p^\wedge\otimes\mathbb Q]
        \]
        is surjective and the natural map
        \[
            [X,\Sigma Y]\to[X,\Sigma Y_p^\wedge]
        \]
        is injective.
    \end{proposition}
    \begin{proof}
        We first show that the two conclusions are equivalent. From the long exact sequence induced by the arithmetic fracture square of $Y$, the map
        \[
        [X,Y_p^\wedge]\oplus[X,Y\otimes\mathbb Q] \to [X,Y_p^\wedge\otimes\mathbb Q]
        \]
        is surjective if and only if
        \[
            [X,\Sigma Y]\to[X,\Sigma Y_p^\wedge]\oplus[X,\Sigma Y\otimes\mathbb Q]
        \]
        is injective. This is equivalent to
        \[
            [X,\Sigma Y]\to[X,\Sigma Y_p^\wedge]
        \]
        being injective by using the next arrow in the long exact sequence and the fact that
        \[
            [X,\Sigma Y\otimes\mathbb Q]\to[X,\Sigma Y_p^\wedge\otimes\mathbb Q]
        \]
        is injective.

        Next, we shall show that the map
        \[
            [X,Y_p^\wedge]\to[X,Y_p^\wedge\otimes\mathbb Q]/[X,Y\otimes\mathbb Q]
        \]
        is surjective. The argument is similar to the one for Proposition \ref{prop:AParg}. By taking suspensions if necessary, let us assume that $X$ and $Y$ are connective. For each $i$, let us fix an identification
        \begin{align*}
        \pi_i(X) &= \mathbb Z_{(p)}^{\oplus s_i}\oplus\text{(torsion)}\\
            \pi_i(Y) &= \mathbb Z_{(p)}^{\oplus t_i}\oplus\text{(torsion)}.
            \end{align*}
        Then, an element
        \[
            f\in [X,Y_p^\wedge\otimes\mathbb Q]/[X,Y\otimes\mathbb Q]
        \]
        is a collection consisting of a $s_i\times t_i$ matrix, say $\pi_i(f)$, valued in $\mathbb Q_p/\mathbb Q$ for each $i$.

        Let $g_{-1}$ be the zero map from $X$ to $Y_p^\wedge$. We inductively define $g_k$ as follows. Let $N$ be a large integer and let $B$ be a $\mathbb Z_p$-valued matrix such that
        \[
            \pi_k(g_{k-1})+p^{N+k}B = \pi_k(f) \pmod{\mathbb Q}.
        \]
        This is possible (for any $N$) since the natural map $p^{N+k}\mathbb Z_p\to\mathbb Q_p/\mathbb Q$ is surjective. Since $(X,Y)$ is preAP, if $N$ was chosen large enough, then there is a map $h_k:X\to Y_p^\wedge$ such that $\pi_i(h_k)$ is zero for $i<k$ and $\pi_k(h_k)=p^NB$. Finally, we define
        \[
            g_k = g_{k-1} + p^kh_k.
        \]
        Then, the limit of $g_k$, i.e.
        \[
            \sum_{k\geq 0}p^k h_k
        \]
        exists in $[X,Y_p^\wedge]$ and it maps to $f$ by construction.
    \end{proof}
    
    \section{Uniqueness of \textp-local \textbpn}\label{sec:uniqueplocal}
        The goal of this section is to prove Theorem \ref{thm:introplocal}, i.e. to recover the homotopy type of $\BP\langle n\rangle$ from that of $\BP\langle n\rangle_p^\wedge$.
    \subsection{Action of Ext on homotopy groups}
    Recall from Example \ref{ex:KD} that we have
    \[
        \Ext_{\mathcal A}(\mathcal A//E(n),\mathbb F_p) = \mathbb F_p[v_0,\dots,v_n]
    \]
    for all $n$. Therefore, given an Ext class
    \[
        f\in \Ext_{\mathcal A}^{s,t}(\mathcal A//E(m),\mathcal A//E(n)),
    \]
    it induces a map of $\mathbb F_p$-vector spaces
    \[
        f_\ast : \mathbb F_p[v_0,\dots,v_n]\to \mathbb F_p[v_0,\dots,v_m]
    \]
    given by Yoneda product of Ext classes, shifting the bidegrees by $(s,t)$.
    
    The goal of this section is to compute $f_\ast$ for all Ext classes $f$ and prove Theorem \ref{thm:intro5}. Similar to notes made in Example \ref{ex:KD}, it can be shown that $(v_i f)_\ast = v_i f_\ast$ for all $0\leq i\leq m$.

    If we let
    \[
        E(\infty) = \Lambda(Q_0,Q_1,\dots)\subseteq \mathcal A
    \]
    be the infinite exterior sub-Hopf algebra of $\mathcal A$, then everything in Section \ref{ssec:Ext} applies. In particular, we have
    \[
        \Ext_{\mathcal A}(\mathcal A//E(\infty),\mathbb F_p) = \mathbb F_p[v_0,v_1,\dots].
    \]
    Then, consider the zig-zag
    \[
    \begin{tikzcd}
        \Ext_{\mathcal A}(\mathcal A//E(m),\mathcal A//E(n))\ar[d,"\eta_1"]\\
        \Ext_{\mathcal A}(\mathcal A//E(m),\mathcal A//E(\infty))\\
        \Ext_{\mathcal A}(\mathcal A//E(\infty),\mathcal A//E(\infty))\ar[u,"\eta_2"]
    \end{tikzcd}
    \]
    induced by the obvious maps $\mathcal A//E(n)\to \mathcal A//E(\infty)$ and $\mathcal A//E(m)\to\mathcal A//E(\infty)$. If $\eta_1(f)=\eta_2(g)$ for some
    \begin{align*}
        f&\in \Ext_{\mathcal A}(\mathcal A//E(m),\mathcal A//E(n))\\
        g&\in \Ext_{\mathcal A}(\mathcal A//E(\infty),\mathcal A//E(\infty))
    \end{align*}
    then there is a commutative diagram
    \[
    \begin{tikzcd}
        \mathbb F_p[v_0,v_1,\dots]\ar[r,"g_\ast"]\ar[d]&\mathbb F_p[v_0,v_1,\dots]\ar[d]\\
        \mathbb F_p[v_0,\dots, v_n]\ar[r,"f_\ast"]&\mathbb F_p[v_0,\dots,v_m]
        \end{tikzcd}
    \]
    where the vertical maps are the obvious quotient maps. Therefore, we see that $f_\ast$ is determined by $g_\ast$. Moreover, given $f$, there is always a $g$ such that $\eta_1(f)=\eta_2(g)$ since $\eta_2$ is surjective. Indeed, we have
    \[
        \Ext_{\mathcal A}(\mathcal A//E(m),\mathcal A//E(\infty)) = \Ext_{E(m)}(\mathbb F_p,\mathcal A//E(\infty)) = (\mathcal A//E(\infty))[v_0,\dots,v_m]
    \]
    since the action of $E(m)$ on $\mathcal A//E(\infty)$ is trivial. Similarly, we have
    \[
        \Ext_{\mathcal A}(\mathcal A//E(\infty),\mathcal A//E(\infty))=(\mathcal A//E(\infty))[v_0,v_1,\dots]
    \]
    and $\eta_2$ is the quotient map.

    Therefore, it is enough to consider the case $m=n=\infty$, i.e. given a class
    \[
        f\in\Ext_{\mathcal A}(\mathcal A//E(\infty),\mathcal A//E(\infty))=(\mathcal A//E(\infty))[v_0,v_1,\dots],
    \]
    we shall compute the induced map
    \[
        f_\ast :\mathbb F_p[v_0,v_1,\dots]\to \mathbb F_p[v_0,v_1,\dots].
    \]
    \begin{definition}
        Let $\Seq(L)$ be the collection of all sequences $(r_0,r_1,\dots)$ of nonnegative integers such that
        \[
            \sum_{i\geq 0}r_i=L.
        \]
        Given two sequences $R=(r_0,r_1,\dots),S=(s_0,s_1,\dots)\in\Seq(L)$, we consider the following combinatorial situation.

        Suppose there are $L$ boxes that are linearly ordered, and each box has a nonnegative integer written on it. The first $r_0$ boxes have $0$ written on them, the next $r_1$ boxes have $1$ on them, and so on. Also, suppose there are $L$ balls, each with a nonnegative integer written on it, so that there are exactly $s_i$ balls that have $i$ written on them. We assume that two balls with the same integer on them cannot be distinguished, but two boxes with the same integer on them can be distinguished by their linear order.

        Then, a \emph{majorization class from $S$ to $R$} is a method of placing the balls into the boxes, so that each box contains exactly one ball and the integer written on each box is greater than or equal to the integer on the ball that the box contains. Let us write $\Maj(S,R)$ for the collection of all majorization classes.

        Suppose we are given a majorization class $\sigma\in\Maj(S,R)$, and suppose that a ball with $i$ written on it is contained in a box with $i+j$ written on it. Then, we define the \emph{local index} at the box to be
        \[
            \operatorname{locind}(i,i+j) = \begin{cases}
            (0,0,\dots)&\text{if }j=0\\
            p^i\mathbf e_j&\text{if }j>0
            \end{cases}
        \]
        as an element of $\mathbb Z^{\oplus\infty}$. The \emph{index} of $\sigma$, denoted by $\ind(\sigma)$ is the sum of the local indices over all boxes.
    \end{definition}
    \begin{proposition}\label{prop:Extaction}
        Consider the map
        \[
            f_\ast:\mathbb F_p[v_0,v_1,\dots]\to\mathbb F_p[v_0,v_1,\dots]
        \]
        induced by the class
        \[
            f:=P^I \in \Ext_{\mathcal A}^{0,\ast}(\mathcal A//E(\infty),\mathcal A//E(\infty)),
        \]
        where $I=(i_1,i_2,\dots)$. Let $x:=v_0^{r_0}v_1^{r_1}\cdots$ be a monomial.
        \begin{enumerate}
            \item If
            \[
                \deg(x)+\deg(P^I) <0,
            \]
            then $f_\ast(x)=0$.
            \item If
            \[
                \deg(x)+\deg(P^I)=0,
            \]
            then $f_\ast(x)$ is zero unless $i_k=r_k$ for all $k\geq 1$, in which case we have
            \[
                f_\ast(x) = v_0^{r_0+r_1+r_2+\cdots}.
            \]
        \end{enumerate}
        Here, $\deg$ denotes the topological degree.
    \end{proposition}
    \begin{proof}
        The assertion (a) is clear by considering the topological degrees, since there is no element in $\mathbb F_p[v_0,v_1,\dots]$ with negative topological degree.
        
        Applying Definition \ref{def:resolution} to $E(\infty)$ and base-changing along $E(\infty)\to \mathcal A$, we obtain a free resolution $K^\bullet$ of $\mathcal A//E(\infty)$ as left $\mathcal A$-modules described as follows. It is freely generated by symbols $u^R$ where $R=(r_0,r_1,\dots)$ is a sequence of nonnegative integers. The differential is given by
        \[
            d(u^R)=\Sigma_{i=0}^\infty Q_i u^{R-\mathbf f_i},
        \]
        and the free module is made into a cochain complex of $\mathcal A$-modules by requiring that $u^R$ is in cochain level $\sum(-r_i)$, i.e. we have a resolution
        \[
            \cdots\to K^{-2}\to K^{-1}\to K^0\to\mathcal A//E(\infty)\to 0
        \]
        such that
        \begin{align*}
            K^0 &=\mathcal A u^{(0,0,\dots)}\\
            K^{-1}&= \mathcal A u^{(1,0,\dots)}\oplus\mathcal A u^{(0,1,\dots)}\oplus\cdots
        \end{align*}
        and so on.

        We shall construct a map of chain complexes of left $\mathcal A$-modules
        \[
            \widetilde{P^I}:K^\bullet\to K^\bullet
        \]
        that lifts the map of left $\mathcal A$-modules
        \[
            P^I:\mathcal A//E(\infty)\to \mathcal A//E(\infty)
        \]
        that sends $1$ to $P^I$. Given a sequence $S\in\Seq(L)$ of nonnegative integers with sum $L$, let us define $\widetilde{P^I}$ by the formula
        \[
            \widetilde{P^I}(u^S) = \sum_{R\in\Seq(L)} \,\sum_{\sigma\in\Maj(S,R)} P^{I-\ind(\sigma)} u^{R}.
        \]
        As usual, if $I-\ind(\sigma)$ contains negative entries, then the whole term should be considered as zero. Let us check that this map commutes with the differential. Applying the differential later, we have
        \[
            (d\circ \widetilde{P^I})(u^S) = \sum_{R\in\Seq(L)} \,\sum_{\sigma\in\Maj(S,R)}\,\sum_{i=0}^\infty P^{I-\ind(\sigma)}Q_i u^{R-\mathbf f_i}.
        \]
        On the other hand, we have
        \begin{align*}
            &(\widetilde{P^I}\circ d)(u^S)\\
            &=\widetilde{P^I}\left(\sum_{i=0}^\infty Q_i u^{S-\mathbf f_i}\right)\\
            &=\sum_{i=0}^\infty  \sum_{R\in\Seq(L-1)} \,\sum_{\sigma\in\Maj(S-\mathbf f_i,R)} Q_iP^{I-\ind(\sigma)}u^{R}\\
            &=
            \sum_{i=0}^\infty  \sum_{R\in\Seq(L-1)} \,\sum_{\sigma\in\Maj(S-\mathbf f_i,R)}\left(P^{I-\ind(\sigma)}Q_iu^{R}+\sum_{j=1}^\infty P^{I-\ind(\sigma)-p^i\mathbf e_j}Q_{i+j}u^{R}\right).
        \end{align*}

        We would like to find a bijection of the terms between the two sums. Equivalently, we need to find a bijection between the set of tuples
        \[
            \{(R,\sigma,i):R\in\Seq(L),\sigma\in\Maj(S,R),i\in\mathbb Z_{\geq 0}\}
        \]
        and
        \[
            \{(i,R,\sigma,j):i,j\in\mathbb Z_{\geq0},R\in\Seq(L-1),\sigma\in\Maj(S-\mathbf f_i,R)\}
        \]
        so that $P^{I-\ind(\sigma)}Q_iu^{R-\mathbf f_i}$ corresponds to $P^{I-\ind(\sigma)-p^i\mathbf e_j}Q_{i+j}u^R$ (or $P^{I-\ind(\sigma)}Q_iu^R$ if $j=0$). The bijection is given by sending $(R,\sigma,i)$ to $(\ell,R-\mathbf f_i,\sigma\backslash i, i-\ell)$, where $\ell$ is the integer on the ball that is contained in the rightmost box with $i$ written on it in the majorization class $\sigma$, and $\sigma\backslash i$ is a majorization class from $S-\mathbf f_\ell$ to $R-\mathbf f_i$ obtained by removing that rightmost box and the ball that is contained in it.
        
        The above bijection shows that $\widetilde{P^I}$ is a map of cochain complexes. We can also check that
        \[
            \deg(\widetilde{P^I}(u^R)) = \deg(P^I)+\deg(u^R)
        \]
        as expected.

        Let $R=(r_0,r_1,\dots)\in\Seq(L)$ be a sequence with sum $L$. The class
        \[
            x=v_0^{r_0}v_1^{r_1}\cdots\in\Ext_{\mathcal A}(\mathcal A//E(\infty),\mathbb F_p)
        \]
        can be described by the map
        \[
            K^{\bullet}\to \mathbb F_p[-L]
        \]
        that sends $u^{R}$ to $1$ and other $u^S$'s to $0$. Therefore, to compute $(P^I)_\ast (x)$, we need to find the image of $u^S$'s under the composition
        \[
            K^{\bullet}\xrightarrow{\widetilde{P^I}}K^{\bullet}\xrightarrow{x}\mathbb F_p[-L].
        \]
        In particular, we wish to find all sequences $S\in\Seq(L)$ and majorization classes $\sigma\in\Maj(S,R)$ such that $I=\ind(\sigma)$. However, since $\widetilde{P^I}$ shifts the topological degree by $\deg(P^I)$, we would also have
        \[
            \deg(P^I) + \deg(u^S) = \deg(P^{I-\ind(\sigma)})+\deg(u^R) = \deg(u^R).
        \]

        Now to prove (b), if we assume $\deg(x)+\deg(P^I)=0$, then since $\deg(u^R) = -\deg(x)$, we must have $\deg(u^S)=0$. This implies that $S=(L,0,0,\dots)$ since otherwise the degree of $u^S$ will be strictly negative. Also, if $S=(L,0,\dots)$, there is a unique majorization class $\sigma\in\Maj(S,R)$ since all balls look alike, and we can check that
        \[
            \ind(\sigma) = (r_1,r_2,\dots).
        \]
        Therefore, we have that $(P^I)_\ast(x)$ is $v_0^L$ if $I=(r_1,r_2,\dots)$ and zero otherwise.
    \end{proof}
    \subsection{Uniqueness of \textp-local \textbpn}
    In this section, we shall prove that $(\BP\langle n\rangle, \BP\langle m\rangle)$ is AP if $(p,m,n)$ satisfies the vanishing line hypothesis. Let us first solve the algebraic version of the problem.
    \begin{proposition}\label{prop:algAP}
        Let $n\geq m$ be any nonnegative integers. Then, for any integers $D,k\in\mathbb Z$, the map
        \begin{align*}
            &\left(\bigoplus_{s=0}^\infty \Ext^{s,s+D}_{\mathcal A}(\mathcal A//E(m),\mathcal A//E(n))\right)[v_0^{-1}]\\
            & \to\prod_{c\leq k} \Hom_{\mathbb F_p[v_0^{\pm}]}((\mathbb F_p[v_0^{\pm},v_1,\dots,v_n])_c,(\mathbb F_p[v_0^{\pm},v_1,\dots,v_n])_{c+D})
        \end{align*}
        given by $f\mapsto f_\ast$ is surjective. Here, we grade $\mathbb F_p[v_0^{\pm},v_1,\dots,v_n],\mathbb F_p[v_0^{\pm},v_1,\dots,v_m]$ by topological degrees and $(-)_c$ denotes the $c$-th graded piece.
    \end{proposition}
    \begin{proof}
        Let $I=(i_1,i_2,\dots, i_n)$ be any sequence of nonnegative integers of length $n$. We shall prove that there is a class
        \[
            \phi^I\in\Ext_{\mathcal A}(\mathcal A//E(m),\mathcal A//E(n))
        \]
        that can be represented by a cycle of the form
        \[
            v_0^N P^I + v_0^{N-1}x_{N-1}+\cdots+x_0\in(\mathcal A//E(n))[v_0,\dots,v_m]
        \]
        for some $x_0,\dots,x_{N-1}\in(\mathcal A//E(n))[v_1,\dots,v_m]$ and sufficiently large $N$. Recall that the differential is
        \[
            d = v_0Q_0+\cdots+v_mQ_m = v_0Q_0 + d^1.
        \]
        Since $Q_0P^I=0$, we have $Q_0d^1P^I=0$, so that by Proposition \ref{prop:exact}, there is some $x_{N-1}$ such that $Q_0x_{N-1} = d^1P^I$. Then, we also have
        \[
        Q_0d^1 x_{N-1} = -d^1Q_0 x_{N-1}  = -d^1 d^1P^I =0 
        \]
        so that we can define $x_{N-2}$ such that $Q_0x_{N-2} =d^1x_{N-1}$. Inductively, if we define $x_i$ by
        \[
            Q_0x_i = d^1  x_{i+1}
        \]
        then we have
        \[
            Q_0 d^1x_i = -d^1Q_0x_i = -d^1d^1x_{i+1}=0
        \]
        and we can continue.

        We claim that at some point, we get $x_i=0$, given that we choose sufficiently large $N$. Note that by construction, $x_i$ has Adams degree $N-i$ and has topological degree $\deg(P^I)$. If
        \[
            v_1^{r_1}\cdots v_m^{r_m}x\in(\mathcal A//E(n))[v_0,\dots,v_m]
        \]
        is an element of Adams degree $N-i$ in the expansion of $x_i$ into monomials, then we have
        \[
            \deg(v_1^{r_1}\cdots v_m^{r_m})\geq (2p-2)(N-i)
        \]
        so that
        \[
            \deg(x) \leq \deg(P^I) - (2p-2)(N-i)
        \]
        where $\deg$ is the topological degree. If $w$ is the weight of $P^I$, then by Lemma \ref{lem:maxdeg}, we have
        \[
            -w+1 \leq \deg(x)\leq \deg(P^I)-(2p-2)(N-i).
        \]
        Therefore, if we start with a sufficiently large $N$, then there is some $i$ such that the above inequality cannot hold, which means that we must have $x_i=0$.

        If we have $x_i=0$, then we have
        \begin{align*}
            &d(v_0^NP^I+v_0^{N-1}x_{N-1}+\cdots+v_0^{i+1}x_{i+1})\\
            &= v_0^{N+1}Q_0P^I + (v_0^Nd^1P^I + v_0^N Q_0 x_{N-1}) + (v_0^{N-1}d^1x_{N-1}+v_0^{N-1}Q_0x_{N-2})+\cdots = 0
        \end{align*}
        and we can name this class $\phi^I$.

        Next, let us analyze the map
        \[
            (\phi^I)_\ast:\mathbb F_p[v_0,\dots,v_n]\to\mathbb F_p[v_0,\dots,v_m].
        \]
        By the discussion at the beginning of Section \ref{sec:uniqueplocal}, it is enough to consider the case $m=n=\infty$. Suppose that
        \[
            x = v_1^{r_1}v_2^{r_2}\cdots
        \]
        is a monomial such that $\deg(x)+\deg(P^I)\leq 0$. If
        \[
            v_0^{s_0}v_1^{s_1}\cdots P^J \in(\mathcal A//E(\infty))[v_0,v_1,\dots]
        \]
        is a monomial in the expansion of
        \[\phi^I = v_0^NP^I+\cdots.
        \]
        Then, unless $(s_0,s_1,\dots)=(N,0,\dots)$, we have
        \[
            \deg(P^I) = \deg(v_0^{s_0}v_1^{s_1}\cdots) + \deg(P^J) - \deg(v_0^N)> \deg(P^J)
        \]
        so that $\deg(x)+\deg(P^J)<0$, which implies that $(P^J)_\ast (x) =0$ by Proposition \ref{prop:Extaction}. Therefore, we have
        \[
            (\phi^I)_\ast(x) = v_0^N(P^I)_\ast(x)
        \]
        and by Proposition \ref{prop:Extaction}, we have
        \[
            (\phi^I)_\ast(x) = \begin{cases}
                v_0^{N+i_1+i_2+\cdots}&\text{if }x=v_1^{i_1}v_2^{i_2}\cdots\\
                0&\text{otherwise}
            \end{cases}
        \]
        assuming $\deg(x)+\deg(P^I)\leq 0$.
		
		Suppose that $D$ and $k$ are fixed integers as in the statement of the Proposition. For each integer $c\in\mathbb Z_{\geq0}$, let $M_c$ be the collection of all sequences $(i_1,\dots,i_n,j_1,\dots,j_m)$ of nonnegative integers of length $n+m$ such that
		\[
			\deg(v_1^{i_1}\cdots v_n^{i_n}) = \deg(v_1^{j_1}\cdots v_m^{j_m}) - D = c.
		\]
		Let us give a lexicographical ordering on $M_d$ and let us order
		\[
			M = \bigsqcup_{c=0}^D M_c
		\]
		by declaring that elements of $M_i$ are smaller than elements of $M_j$ for $i<j$.
		
		There is a natural basis of
		\[
			\prod_{c\leq k} \Hom_{\mathbb F_p[v_0^{\pm}]}((\mathbb F_p[v_0^{\pm},v_1,\dots,v_n])_c,(\mathbb F_p[v_0^{\pm},v_1,\dots,v_n])_{c+D})
		\]
		indexed by $M$.  Namely, for each $(i_1,\dots,i_n,j_1,\dots,j_m)$, we take the $\mathbb F_p[v_0^{\pm}]$-linear map sending $v_1^{i_1}\cdots v_n^{i_m}$ to $v_1^{j_1}\cdots v_m^{j_m}$ and other monomials to zero.
		
		We can also define an $M$-indexed subset of
		\[
			\left(\bigoplus_{s=0}^\infty \Ext^{s,s+D}_{\mathcal A}(\mathcal A//E(m),\mathcal A//E(n))\right)[v_0^{-1}],
		\]
		by taking the class represented by the cycle
		\[
		v_1^{j_1}\cdots v_m^{j_m}\phi^{(i_1,\dots,i_n)}\in\mathcal (A//E(n))[v_0,\dots,v_m]
		\]
		for each $(i_1,\dots,i_n,j_1,\dots,j_m)\in M$. Then, by the previous discussion on $(\phi^I)_\ast$, the representation of the image of this set by the basis in the previous paragraph is a lower triangular matrix where diagonal entries are powers of $v_0$. Therefore, the matrix is nonsingular over $\mathbb F_p[v_0^{\pm}]$, and we have the desired surjectivity.
    \end{proof}
    \begin{theorem}\label{thm:BPnAP}
        Let $n\geq m$ be nonnegative integers such that $(m,n)$ satisfies the vanishing line hypothesis (resp. weak vanishing line hypothesis). Then, the pair $(\BP\langle n\rangle, \BP\langle m\rangle)$ is AP (resp. preAP).
    \end{theorem}
    \begin{proof}
        Let us first assume that $(m,n)$ satisfies the vanishing line hypothesis. We wish to show that $(\BP\langle n\rangle, \Sigma^{-D}\BP\langle m\rangle)$ is preAP for all $D\in\mathbb Z$. If $D$ is an odd number, then there is nothing to prove since
        \[
        	[\BP\langle n\rangle_p^\wedge,\tau_{\leq k}\Sigma^{-D}\BP\langle m\rangle_p^{\wedge}]\otimes_{\mathbb Z}\mathbb Q=0
        \]
        for all $k$. Therefore, let us assume that $D$ is even.
        
        Let $k$ be a fixed integer. We wish to show that
        \[
            [\BP\langle n\rangle_p^\wedge,\Sigma^{-D}\BP\langle m\rangle_p^\wedge]
            \to \prod_{c\leq k}\Hom_{\mathbb Q_p}(\pi_c(\BP\langle n\rangle_p^\wedge)\otimes_{\mathbb Z}\mathbb Q,\pi_{c+D}(\BP\langle m\rangle_p^\wedge)\otimes_{\mathbb Z}\mathbb Q)
        \]
        is surjective.
        Using Proposition \ref{prop:algAP}, let us choose some classes
        \[
        	\phi_1,\phi_2,\dots\in \bigoplus_{s=0}^\infty \Ext^{s,s+D}_{\mathcal A}(\mathcal A//E(m),\mathcal A//E(n))
        \]
        such that their images in
        \[
        	\prod_{c\leq k} \Hom_{\mathbb F_p[v_0^{\pm}]}((\mathbb F_p[v_0^{\pm},v_1,\dots,v_n])_c,(\mathbb F_p[v_0^{\pm},v_1,\dots,v_n])_{c+D})
        \]
        form a basis.
        
        First, we shall prove that $\phi_i$ is a permanent cycle in the Adams spectral sequence
        \[
        	E_2^{s,t} = \Ext_{\mathcal A}^{s,t}(\mathcal A//E(m),\mathcal A//E(n))\Rightarrow [\Sigma^{t-s}\BP\langle n\rangle_p^\wedge,\BP\langle m\rangle_p^\wedge]
        \]
        for each $i$, possibly after multiplying it by a power of $v_0$. Suppose that 
       	\[
       		\phi_i\in \Ext_{\mathcal A}^{s,s+D}(\mathcal A//E(m),\mathcal A//E(n)).
       	\]
       	Then, $v_0^N\phi_i\in\Ext^{s+N,s+N+D}$, so a target of a differential from $v_0^N\phi_i$ would have to be detected by $\Ext^{s+N+r, s+N+D+r-1}$ in the $E_2$-page for some $r\geq 2$. Since $D-1$ is odd, this group is zero if
       	\[
       		a(s+N+r) + D-1 > 1-2p^{n+1}
       	\]
       	by the assumption that $(p,m,n)$ satisfies the vanishing line hypothesis. If $N$ is sufficiently large, this inequality holds for all $r\geq2$, so that $v_0^N\phi_i$ is a permanent cycle.
       	
       	Therefore, by multiplying powers of $v_0$ if necessary, we may assume that $\phi_i$ is a permanent cycle in the Adams spectral sequence for all $i$. Also, since $\{\phi_1,\phi_2,\dots\}$ is a finite set, by multiplying powers of $v_0$ if necessary, we may assume that the Adams degree of $\phi_i$ is $s$ for all $i$.
       	
	    Since $\phi_i$ is a permanent cycle in the Adams spectral sequence, it lifts to a map
       	\[
       	\widetilde{\phi_i}:\BP\langle n\rangle_p^\wedge\to\Sigma^{-D}\BP\langle m\rangle_p^\wedge
       	\]
       	detecting it. Also, since the Adams spectral sequence for $\pi_\ast(\BP\langle n\rangle_p^\wedge)$ and $\pi_\ast(\BP\langle n\rangle_p^\wedge)$ degenerates, the map
       	\[
       		\pi_\ast(\widetilde{\phi_i}):\mathbb Z_p[v_1,\dots,v_n]\to\Sigma^{-D}\mathbb Z_p[v_1,\dots,v_n]
       	\]
       	induces
        \[
            (\phi_i)_\ast:\mathbb F_p[v_0,\dots,v_n]\to \Sigma^{-D}\mathbb F_p[v_0,\dots,v_m]
        \]
        by taking the associated graded of the Adams filtration.

        Since the $\mathbb F_p[v_0^\pm]$-dimension of
        \[
\prod_{c\leq k} \Hom_{\mathbb F_p[v_0^{\pm}]}((\mathbb F_p[v_0^{\pm},v_1,\dots,v_n])_c,(\mathbb F_p[v_0^{\pm},v_1,\dots,v_n])_{c+D})
        \]
        is equal to the $\mathbb Q_p$-dimension of
        \[
            \prod_{c\leq k}\Hom_{\mathbb Q_p}(\pi_c(\BP\langle n\rangle_p^\wedge)\otimes_{\mathbb Z}\mathbb Q,\pi_{c+D}(\BP\langle m\rangle_p^\wedge)\otimes_{\mathbb Z}\mathbb Q),
        \]
        it is enough to see that the image of $\{\widetilde{\phi_1},\widetilde{\phi_2},\dots\}$ in this $\mathbb Q_p$-vector space is linearly independent, which is equivalent to saying that it is $\mathbb Z_p$-linearly independent in
        \[
            \prod_{c\leq k}\Hom_{\mathbb Z_p}(\pi_c(\BP\langle n\rangle_p^\wedge),\pi_{c+D}(\BP\langle m\rangle_p^\wedge)).
        \]
        If we assume that there is a relation $\sum \alpha_i\widetilde{\phi_i}=0$ in the above $\mathbb Z_p$-module, then since it is a free module, we may divide by powers of $p$ if necessary and assume that at least one $\alpha_i$ is a $p$-adic unit. Then, we obtain a contradiction by passing to the associated graded with respect to the Adams filtration, since $\{(\phi_1)_\ast,(\phi_2)_\ast,\dots\}$ is linearly independent over $\mathbb F_p$.

        If $(m,n)$ satisfies the weak vanishing line hypothesis, we can modify the proof above to only care about the case $D=0$. In this case, the inequality that we have used
        \[
            a(s+N+r)+D-1 > 1-2p^{n+1}
        \]
        is always true with $a=0$ and $D=0$. Following the same proof, we obtain that $(\BP\langle n\rangle,\BP\langle m\rangle)$ is preAP.
    \end{proof}
    \begin{corollary}\label{cor:tmfpr}
        Let $X$ be a $p$-local spectrum of finite type. If $X_p^\wedge\simeq \tmf_p^\wedge$, then $X\simeq \tmf_{(p)}$.
    \end{corollary}
    \begin{proof}
        Combine Theorem \ref{thm:BPnAP} and Proposition \ref{prop:tmfBP2}.
    \end{proof}
    \begin{corollary}\label{cor:uniqueplocal}
        Suppose either that $p$ is odd or that $p=2$ and $n\leq 2$. Then, the $\mathbb F_p$-cohomology of $\BP\langle n\rangle$ as an $\mathcal A$-module determines the spectrum $\BP\langle n\rangle$.

        More generally, suppose that $Y$ is a direct sum of even suspensions of $\BP\langle n\rangle$,
        \[
            Y=\Sigma^{2i_1}\BP\langle n\rangle\oplus\cdots\oplus\Sigma^{2i_k}\BP\langle n\rangle
        \]
        such that $0\leq i_1\leq \cdots\leq i_k<p^{n+1}-1$. Then, if $X$ is a $p$-local spectrum of finite type, any isomorphism
        \[
            H^\ast(X;\mathbb F_p)\simeq H^\ast(Y;\mathbb F_p)
        \]
        of $\mathcal A$-modules can be lifted to an equivalence $X\simeq Y$.
    \end{corollary}
    \begin{proof}
        By Theorem \ref{thm:vanish}, the $E_2$-page of the Adams spectral sequence
        \[
            E_2^{s,t}=\Ext_{\mathcal A}^{s,t}(H^\ast(Y;\mathbb F_p),H^\ast(X;\mathbb F_p))\Rightarrow [\Sigma^{t-s}X,Y_p^\wedge]
        \]
        vanishes for $t-s=-1$, since $-1- 2(i_k-i_0) > 1-2p^{n+1}$. Therefore, we obtain an equivalence $X_p^\wedge\simeq Y_p^\wedge$ lifting the isomorphism of cohomology.

        By Proposition \ref{prop:AParg} and Theorem \ref{thm:BPnAP}, we have an equivalence $X\simeq Y$, which lifts $X_p^\wedge/p\simeq Y_p^\wedge/p$ and hence also the isomoprhism of cohomology.
    \end{proof}
    \begin{remark}
        When $(p,n)=3$, we see that Theorems \ref{thm:intro1} and \ref{thm:introplocal} hold by the same argument since $(\BP\langle 3\rangle,\BP\langle3\rangle)$ is preAP.
    \end{remark}
    \begin{corollary}\label{cor:inj}
        Let $p$ be an odd prime and let $n\geq m$ be nonnegative integers. Then, for any integer $k$, the natural map
        \[
            [\BP\langle n\rangle,\Sigma^k\BP\langle m\rangle]\to [\BP\langle n\rangle_p^\wedge,\Sigma^k\BP\langle m\rangle_p^\wedge]
        \]
        is injective. The same is true when $p=2$ if either $n>m$ or $n=m\leq 2$.
    \end{corollary}
    \begin{proof}
        This follows from Theorem \ref{thm:BPnAP} and Proposition \ref{prop:fracture}.
    \end{proof}
    \begin{corollary}\label{cor:tmfsplit}
        When $p>3$, the spectrum of topological modular forms $\tmf_{(p)}$ splits into a direct sum of suspensions of $\BP\langle 2\rangle$.
    \end{corollary}
    \begin{proof}
        It is shown in \cite[Theorem 21.5]{Rezktmf} that $H^\ast(\tmf;\mathbb F_p)$ as an $\mathcal A$-module has a filtration whose associated graded is a sum of even suspensions of $H^\ast(\BP\langle 2\rangle;\mathbb F_p)$ with suspension degrees between $0$ and $2(p^2+p-12)$. However, the extension problems going from the associated graded to the filtered module is trivial since we have proved that all of the relevant $\Ext^1$ groups vanish. Therefore, $H^\ast(\tmf;\mathbb F_p)$ is a direct sum of small even suspensions of $H^\ast(\BP\langle2\rangle;\mathbb F_p)$ and we are done by the previous corollary.
    \end{proof}

    \bibliographystyle{amsalpha}
    \bibliography{ref}

\providecommand{\bysame}{\leavevmode\hbox to3em{\hrulefill}\thinspace}
\providecommand{\MR}{\relax\ifhmode\unskip\space\fi MR }
\providecommand{\MRhref}[2]{%
  \href{http://www.ams.org/mathscinet-getitem?mr=#1}{#2}
}
\providecommand{\href}[2]{#2}
\begin{thebibliography}{Mat16}

\bibitem[AL17]{AL}
Vigleik Angeltveit and John~A. Lind, \emph{Uniqueness of {$BP \langle
  n\rangle$}}, J. Homotopy Relat. Struct. \textbf{12} (2017), no.~1, 17--30.

\bibitem[AP76]{AP}
J.~F. Adams and S.~B. Priddy, \emph{Uniqueness of {$B{\rm SO}$}}, Math. Proc.
  Cambridge Philos. Soc. \textbf{80} (1976), no.~3, 475--509.

\bibitem[BM13]{BM}
Maria Basterra and Michael~A. Mandell, \emph{The multiplication on {BP}}, J.
  Topol. \textbf{6} (2013), no.~2, 285--310.

\bibitem[Boa99]{CCSS}
J.~Michael Boardman, \emph{Conditionally convergent spectral sequences},
  Homotopy invariant algebraic structures ({B}altimore, {MD}, 1998), Contemp.
  Math., vol. 239, Amer. Math. Soc., Providence, RI, 1999, pp.~49--84.

\bibitem[CM15]{CM}
Steven~Greg Chadwick and Michael~A. Mandell, \emph{{$E_n$} genera}, Geom.
  Topol. \textbf{19} (2015), no.~6, 3193--3232.

\bibitem[Dev24]{Sanath}
Sanath~K. Devalapurkar, \emph{Higher chromatic {T}hom spectra via unstable
  homotopy theory}, Algebr. Geom. Topol. \textbf{24} (2024), no.~1, 49--108.

\bibitem[Goe10]{tmfsurvey}
Paul~G. Goerss, \emph{Topological modular forms [after {H}opkins, {M}iller and
  {L}urie]}, Ast\'{e}risque (2010), no.~332, Exp. No. 1005, viii, 221--255.

\bibitem[HK01]{HK}
Po~Hu and Igor Kriz, \emph{Real-oriented homotopy theory and an analogue of the
  {A}dams-{N}ovikov spectral sequence}, Topology \textbf{40} (2001), no.~2,
  317--399.

\bibitem[Mat16]{Htmf}
Akhil Mathew, \emph{The homology of tmf}, Homology Homotopy Appl. \textbf{18}
  (2016), no.~2, 1--29.

\bibitem[Mil58]{Milnor}
John Milnor, \emph{The {S}teenrod algebra and its dual}, Ann. of Math. (2)
  \textbf{67} (1958), 150--171.

\bibitem[Rez07]{Rezktmf}
Charles Rezk, \emph{Supplementary notes for {M}ath 512 (version 0.18)}, 2007,
  https://rezk.web.illinois.edu/512-spr2001-notes.pdf.

\bibitem[Wil75]{WilsonBP}
W.~Stephen Wilson, \emph{The {$\Omega $}-spectrum for {B}rown-{P}eterson
  cohomology. {II}}, Amer. J. Math. \textbf{97} (1975), 101--123.

\end{thebibliography}
\end{document}